\documentclass[11pt]{amsart}

\usepackage{amsmath, amscd, amssymb}
\usepackage[frame,cmtip,arrow,matrix,line,graph,curve]{xy}
\usepackage{graphpap, color, paralist, pstricks}
\usepackage[mathscr]{eucal}
\usepackage[pdftex]{graphicx}
\usepackage[pdftex,colorlinks,backref=page,citecolor=blue]{hyperref}
\usepackage{tikz}

\hypersetup{%
pdftitle={Mori's program for $\bar{M}_{0,7}$ with symmetric divisors},%
pdfauthor={Han-Bom Moon},%
pdfkeywords={moduli of curves, minimal model program, moduli space},%
citecolor=blue,%
linkcolor=blue,%
}

\newtheorem{theorem}{Theorem}[section]
\newtheorem{proposition}[theorem]{Proposition}
\newtheorem{corollary}[theorem]{Corollary}
\newtheorem{lemma}[theorem]{Lemma}

\theoremstyle{definition}
\newtheorem{definition}[theorem]{Definition}
\newtheorem{example}[theorem]{Example}

\newtheorem{question}[theorem]{Question}
\newtheorem{remark}[theorem]{Remark}

\newcommand{\PP}{\mathbb{P}}
\newcommand{\QQ}{\mathbb{Q}}
\newcommand{\CC}{\mathbb{C}}

\newcommand{\ZZ}{\mathbb{Z}}
\newcommand{\NN}{\mathbb{N}}

\newcommand{\cO}{\mathcal{O} }

\newcommand{\cB}{\mathcal{B} }
\newcommand{\cC}{\mathcal{C} }

\newcommand{\cM}{\mathcal{M} }

\newcommand{\cU}{\mathcal{U} }

\newcommand{\spec}{\mathrm{Spec}\;}
\newcommand{\proj}{\mathrm{Proj}\;}
\newcommand{\im}{\mathrm{im}}

\def\Mzn{\overline{\mathrm{M}}_{0,n} }
\def\Mza{\overline{\mathrm{M}}_{0,A} }
\def\Mzf{\overline{\mathrm{M}}_{0,5} }
\def\Mzs{\overline{\mathrm{M}}_{0,6} }
\def\Mzv{\overline{\mathrm{M}}_{0,7} }
\def\Uznprd{\overline{\mathrm{U}}_{0,n}(\PP^{r}, d) }
\def\Uznpdd{\overline{\mathrm{U}}_{0,n}(\PP^{d}, d) }

\def\Uzvptt{\overline{\mathrm{U}}_{0,7}(\PP^{3}, 3) }
\def\Uzzptt{\overline{\mathrm{U}}_{0,0}(\PP^{3}, 3) }
\def\cUznprd{\overline{\cU}_{0,n}(\PP^{r}, d) }
\def\cUznpdd{\overline{\cU}_{0,n}(\PP^{d}, d) }

\def\cUzzptt{\overline{\cU}_{0,0}(\PP^{3}, 3) }
\def\Mznpdd{\overline{\mathrm{M}}_{0,n}(\PP^{d}, d) }
\def\Mzvptt{\overline{\mathrm{M}}_{0,7}(\PP^{3}, 3) }
\def\Mzzptt{\overline{\mathrm{M}}_{0,0}(\PP^{3}, 3) }
\def\cMznpdd{\overline{\cM}_{0,n}(\PP^{d}, d) }

\def\cMzzptt{\overline{\cM}_{0,0}(\PP^{3}, 3) }

\def\cMg{\overline{\cM}_{g} }
\def\SL{\mathrm{SL}}
\def\git{/\!/ }

\begin{document}

\title{Mori's program for $\Mzv$ with symmetric divisors}
\date{\today}
\author{Han-Bom Moon}
\address{Department of Mathematics, Fordham University, Bronx, NY 10458}
\email{hmoon8@fordham.edu}

\begin{abstract}
We complete Mori's program with symmetric divisors for the moduli space of stable seven pointed rational curves. We describe all birational models in terms of explicit blow-ups and blow-downs. We also give a moduli theoretic description of the first flip, which have not appeared in literature.
\end{abstract}

\maketitle

%%%%%%%%%%%%%%%%%%%%%%%%%%%%%%%%%%%

\section{introduction}

The aim of this paper is running \textbf{Mori's program} for $\Mzv$, the moduli space of stable seven-pointed rational curves. Mori's program, a minimal model program for a given moduli space $M$, consists of following: 1) Compute the cone of effective divisors $\mathrm{Eff}(M)$ for $M$ and the chamber structure on it, so called the stable base locus decomposition. 2) For an effective divisor $D$ we may compute a projective model 
\[
	M(D) := \proj \bigoplus_{m \ge 0}H^{0}(M, \cO(mD))
\]
with a rational contraction $M \dashrightarrow M(D)$. Because any rational contraction is obtained in this way (\cite{HK00}), by running Mori's program we are able to classify all birational models of $M$ which are simpler than $M$. Furthermore, since $M$ is a moduli space, we may expect that some of $M(D)$ also have certain good moduli theoretic interpretations. 

Since Hassett and Hyeon initiated the study of birational geometry of moduli spaces of stable curves in a viewpoint toward Mori's program in \cite{Has05, HH09, HH13}, there has been a great amount of success and progress in this direction. Although the initial motivation, finding the (final log) canonical models of moduli spaces of stable curves $\cMg$ succeeded only for a few small genera \cite{Has05, HL10, Fed12, FS13}, but there have constructed many modular birational models of $\cMg$ and they have been studied in a theoretical framework of Mori's program. Also the same framework has been applied to many other moduli spaces for instance Hilbert scheme of points (\cite{ABCH13}) and the moduli space of stable maps (\cite{Che08b, CC10, CC11}). 

We are interested in running Mori's program for $\Mzn$, the moduli space of stable $n$-pointed rational curves. Since $\dim \mathrm{N}^{1}(\Mzn)_{\QQ}$ grows exponentially, it is almost impossible to determine all birational models even for very small $n$. But if we restrict ourselves to the space $\mathrm{N}^{1}(\Mzn)_{\QQ}^{S_{n}}$ of $S_{n}$-invariant divisors (or symmetric divisors), then the dimension grows linearly. Thus we may try to classify all birational models appear in Mori's program at least for small $n$. 

The first non-trivial case is $n = 6$ and it was investigated in \cite{Moo13a}. In this case, there are two divisorial contractions and no flip. These two contractions are classically well-known varieties so called Segre cubic and Igusa quartic. The next case $n = 7$, which we study in this paper, is interesting because there are two flips of $\Mzv$. It seems that in literature, there has been no description of these spaces. 

\subsection{The first main result - Mori's program}

In the first half of this paper, we classify all projective models appear in Mori's program. In this case $\dim \mathrm{N}^{1}(\Mzv)_{\QQ}^{S_{7}} = 2$ and $\mathrm{Eff}(\Mzv)$ is generated by two boundary divisors $B_{2}$ and $B_{3}$. To describe the result in an effective way, we use the interval notation for divisor classes. For two divisor classes $D_{1}$ and $D_{2}$, $[D_{1}, D_{2})$ is the set of all divisor classes $aD_{1} + bD_{2}$ where $a \ge 0$ and $b > 0$. Similarly, we can define $(D_{1}, D_{2})$, $(D_{1}, D_{2}]$, and $[D_{1}, D_{2}]$ as well. All divisor classes below are defined in Section \ref{sec:divcurve}. We describe the flipping locus $B_{2}^{3}$ and $B_{2}^{2}$ later in this section.

\begin{theorem}\label{thm:Moriprogramintro}(Theorem \ref{thm:Moriprogram})
Let $D$ be a symmetric effective divisor of $\Mzv$. Then:
\begin{enumerate}
\item If $D \in (\psi - K_{\Mzv}, K_{\Mzv} + \frac{1}{3}\psi)$, $\Mzv(D) \cong \Mzv$. 
\item If $D \in [K_{\Mzv}+\frac{1}{3}\psi, B_{3})$, $\Mzv(D) \cong \Mza$, the moduli space of weighted pointed stable curves with weight $A = \left(\frac{1}{3}, \cdots, \frac{1}{3}\right)$.
\item If $D = \psi - K_{\Mzv}$, $\Mzv(D)$ is isomorphic to the Veronese quotient $V_{A}^{3}$ where $A = \left(\frac{4}{7}, \cdots, \frac{4}{7}\right)$.
\item If $D \in (\psi - 3K_{\Mzv}, \psi - K_{\Mzv})$, $\Mzv(D) \cong \Mzv^{3}$, which is a flip of $\Mzv$ over $V_{A}^{3}$. The flipping locus is $B_{2}^{3}$.
\item If $D = \psi - 3K_{\Mzv}$, $\Mzv(D)$ is a small contraction of $\Mzv^{3}$.
\item If $D \in (\psi - 5K_{\Mzv}, \psi - 3K_{\Mzv})$, $\Mzv(D) \cong \Mzv^{2}$, which is a flip of $\Mzv^{3}$ over $\Mzv(\psi - 3K_{\Mzv})$. The flipping locus is the proper transform of $B_{2}^{2}$.
\item If $D \in (B_{2}, \psi - 5K_{\Mzv}]$, $\Mzv(D) \cong \Mzv^{1}$, which is a divisorial contraction of $\Mzv^{2}$. The contracted divisor is the proper transform of $B_{2}$.
\item If $D = B_{2}$ or $B_{3}$, $\Mzv(D)$ is a point.
\end{enumerate}
\end{theorem}

Some of these results are already well-known. The birational models in Items (1) through (3) are models appear in \cite{Has03, GJM13} and they have certain moduli theoretic meaning. Also Mori's program for $\Mzn$ for a subcone generated by $K_{\Mzn}$ and $B = \sum B_{i}$ has been intensively studied in \cite{Sim08, FS10, KM11, AS12} for arbitrary $n$. For $n = 7$, this subcone covers Items (1) and (2). Thus the new result is the opposite direction, Items (3) through (7). 

Along this direction, the chain of birational maps $\Mzv \dashrightarrow \Mzv^{3} \dashrightarrow \Mzv^{2} \to \Mzv^{1}$ shows interesting toroidal birational modifications. On $\Mzv$, $B_{2}$ is a simple normal crossing divisor and at most three irreducible components meet together. Let $B_{2}^{i}$ be the union of nonempty intersections of $i$ irreducible components of $B_{2}$. For $\Mzv \dashrightarrow \Mzv^{3}$, $B_{2}^{3}$ is the flipping locus and on $\Mzv^{3}$ no three irreducible components of $B_{2}$ intersect. For $\Mzv^{3} \dashrightarrow \Mzv^{2}$, the flipping locus is the proper transform of $B_{2}^{2}$ and on $\Mzv^{2}$, irreducible components of $B_{2}$ are disjoint. Finally, on $\Mzv^{2} \to \Mzv^{1}$, the modified locus is the proper transform of $B_{2}^{1}=B_{2}$, the disjoint union of irreducible components and it is a divisorial contraction.

Very recently, Castravet and Tevelev proved in \cite{CT13b} that $\Mzn$ is not a Mori dream space if $n$ is large. However, since the effective cone of $\Mzn/S_{n}$ is simplicial and generated by boundary divisors $B_{i}$ for $2 \le i \le \lfloor \frac{n}{2}\rfloor$, it is believed that $\Mzn/S_{n}$ is a Mori dream space. Because Mori's program of $\Mzn$ with symmetric divisors can be identified with that of $\Mzn/S_{n}$ (\cite[Lemma 6.1]{Moo13a}), we obtain the following result. 

\begin{corollary}
The $S_{7}$-quotient $\Mzv/S_{7}$ is a Mori dream space.
\end{corollary}

In general, we expect that the symmetric cone $\mathrm{Eff}(\Mzn) \cap \mathrm{N}^{1}(\Mzn)_{\QQ}^{S_{n}}$ is in the Mori dream region, so during running Mori's program with symmetric divisors, there is no fundamental technical obstruction. In particular, we expect that the answer for the following question is affirmative.

\begin{question}
For each $2 \le k \le \lfloor \frac{n}{2}\rfloor$, is there a rational contraction $\Mzn \dashrightarrow \mathrm{M}(k)$ which contracts all boundary divisors except $B_{k}$?
\end{question}

For $n \ge 7$, the only previously known such model was $\mathrm{M}(2)$, which is $(\PP^{1})^{n}\git \SL_{2}$ (\cite{KM11}). The space $\Mzv^{1}$ provides $\mathrm{M}(3)$ when $n = 7$.

\subsection{The second main result - Modular interpretation}

So far, all modular birational models of $\overline{\cM}_{g,n}$ have been constructed in two ways. One way is taking GIT quotients of certain parameter spaces, and another way is taking an open proper substack of the stack of all pointed curves. Those two approaches are completely different, but the outcome is essentially moduli spaces of (pointed) curves with worse singularities. For instance, the moduli space $\overline{\cM}_{g}^{ps}$ of pseudostable curves (\cite{Sch91}) can be obtained by allowing cuspidal singularities instead of elliptic tails. By replacing a certain type of subcurves by a cetain type of Gorenstein singularities, we may obtain many other birational models. See \cite{AFS10} for a systematic approach for curves without marked points. Hassett's moduli spaces of weighted stable curves $\overline{\cM}_{g,A}$ are also moduli spaces of semi log canonical pairs (See Section \ref{ssec:Hassettspace}.), so they are moduli spaces of pointed curves with certain types of singularities of pairs as well.

Recently, in \cite{Smy13}, Smyth gave a partial classification of possible modular birational models of $\overline{\cM}_{g,n}$, which are moduli spaces of curves with certain singularity types. When $g = 0$, his result gives a complete classification. One interesting fact is that all of his birational models are contractions of $\Mzn$, because there is no positive dimensional moduli of singularities of arithmetic genus zero. Therefore if one wants to impose a moduli theoretic interpretation of a flip of $\Mzn$, then it must \emph{not} be a moduli space of pointed curves. 
 
In the second half of this paper, we give a moduli theoretic meaning to the first flip $\Mzv^{3}$. The main observation is that both $\Mzv$ and $V_{A}^{3}$ are constructed as GIT quotients (Remark \ref{rem:GITquotient}) and there is a commutative diagram in Figure \ref{fig:commdiagram}.
\begin{figure}[!ht]
\[
	\xymatrix{\overline{\mathrm{M}}_{0,7}(\PP^{3}, 3) \ar[rd] 
	\ar[dd]^{\git\SL_{4}}
	\ar@{<-->}[rr]
	&& \square \ar[ld] \ar[dd]^{\git\SL_{4}}\\
	& I \subset \overline{\mathrm{M}}_{0,0}(\PP^{3}, 3) \times
	(\PP^{3})^{7}\ar[dd]^(0.32){\git\SL_{4}}\\
	\Mzv \ar[rd]\ar@{<-->}[rr] && \Mzv^{3}\ar[ld]\\
	& V_{A}^{3}}
\]
\caption{$\SL_{4}$-quotients of incidence varieties}
\label{fig:commdiagram}
\end{figure}
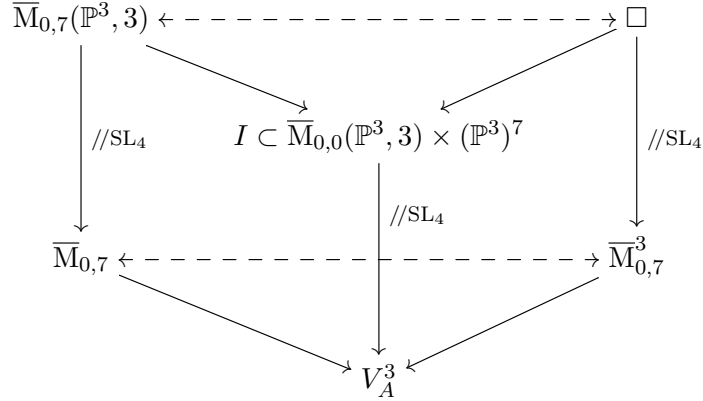

The variety $I$ is the incidence variety in $\Mzzptt \times (\PP^{3})^{7}$, where $\Mzzptt$ is the moduli space of stable maps (\cite{KM94}). All vertical maps are $\SL_{4}$-GIT quotients with certain linearizations (Thus they are not regular maps.). So we may guess that there is a parameter space $X$ in the node $\square$ such that
\begin{enumerate}
\item There is a functorial morphism $X \to \Mzzptt \times (\PP^{3})^{7}$;
\item There is an `incidence variety' $J \subset X$ with $\SL_{4}$-action;
\item With an appropriate linearization, $J\git \SL_{4} \cong \Mzv^{3}$.
\end{enumerate}

Let $\cUznprd$ be the moduli stack of unramified stable maps, introduced in \cite{KKO14}. And let $\Uznprd$ be the coarse moduli space. By analyzing the difference between $\Uzzptt$ and $\Mzzptt$ carefully, we will show that $\Uzzptt \times (\PP^{3})^{7}$ has the role of $X$. 

Unfortunately, there are just few known geometric properties of $\Uzzptt$. For instance, it is not irreducible, and the connectivity and projectivity of the coarse moduli space are unknown. Therefore the standard GIT approach is unavailable. Instead of that, we introduce a `stable locus' $J^{s}$ of $J$ and show that $J^{s}/\SL_{4}$ is a projective variety which is isomorphic to $\Mzv^{3}$. We will denote $J^{s}/\SL_{4}$ by a `formal GIT quotient' $J\git \SL_{4}$ because if we know the projectivity of $\Uzzptt$, then $J^{s}/\SL_{4}$ is indeed isomorphic to $J\git\SL_{4}$ with a standard choice of linearization.

\begin{theorem}\label{thm:modulardescriptionintro}(Theorem \ref{thm:modulardescription})
The formal GIT quotient $J\git \SL_{4}$ is isomorphic to $\Mzv^{3}$.
\end{theorem}

By using this result, we are able to describe a modular description of $\Mzv^{3}$. As we mentioned before, it is not a space of pointed curves anymore. It is a parameter space of data $(C, (x_{1}, x_{2}, \cdots, x_{7}), C')$ where $(C, x_{1}, x_{2}, \cdots, x_{7})$ is an element of $V_{A}^{3}$, which is an arithmetic genus zero pointed curve with certain stability condition (\cite[Theorem 5.1]{GJM13}), and $C'$ is a \textbf{ghost curve}, which is a curve on a non-rigid compactified tangent space $\PP(T_{x}C \oplus \CC)$ for a non-Gorenstein singularity $x \in C$. For the precise definition, see Sections \ref{sec:KKOcompactification} and \ref{sec:modular}. 

The same type of flip appears for Mori's program for all $n \ge 7$ (Remark \ref{rem:largen}). Thus we believe that to run Mori's program for $\Mzn$, it is inevitable to understand the geometry of $\cUznpdd$. We will study geometric properties of this relatively new moduli space in forthcoming papers. 

\subsection{Structure of the paper}

In Section \ref{sec:divcurve} we recall the definitions of several divisor classes and curve classes on $\Mzn$ with their numerical properties. In Section \ref{sec:stablebaselocus}, we compute the stable base locus for every symmetric effective divisor on $\Mzv$. In Section \ref{sec:Moriprogram} we prove Theorem \ref{thm:Moriprogramintro}. Section \ref{sec:KKOcompactification} reviews the moduli space of unramified stable maps and its geometric properties. Finally in Section \ref{sec:modular}, we show Theorem \ref{thm:modulardescriptionintro}. 

We will work over the complex number $\CC$. 

\textbf{Acknowledgement.} The author thanks Q. Chen, A. Gibney, D. Jensen, Y.-H. Kiem, B. Kim, I. Morrison, Y.-G. Oh, D. Swinarski, and J. Tevelev for helpful conversation. Parts of this paper were written while the author was visiting Vietnam Institute for Advanced Study in Mathematics (VIASM). The author wishes to thank VIASM for the invitation and hospitality.

\section{Divisors and curves on $\Mzn$}\label{sec:divcurve}

In this section, we review general facts about divisors and curves on $\Mzn$. All materials in this section is well-known but we leave explicit statements we will use in this paper for reader's convenience. 

\subsection{Divisors on $\Mzn$}

The moduli space $\Mzn$ inherits a natural $S_{n}$ action permuting marked points. A divisor $D$ on $\Mzn$ is called \textbf{symmetric} if it is invariant under the $S_{n}$ action. The Neron-Severi vector space $\mathrm{N}^{1}(\Mzn)_{\QQ}$ has dimension $2^{n-1} - {n \choose 2} - 1$ so the space of divisors on $\Mzn$ is quite huge. But the $S_{n}$-invariant part $\mathrm{N}^{1}(\Mzn)_{\QQ}^{S_{n}} \cong \mathrm{N}^{1}(\Mzn/S_{n})_{\QQ}$ of $\mathrm{N}^{1}(\Mzn)_{\QQ}$ is $\lfloor n/2\rfloor - 1$ dimensional (\cite[Theorem 1.3]{KM96}) so at least for small $n$, computations on the space are doable. 

The following is a list of tautological divisors on $\Mzn$. 

\begin{definition}\label{def:divisorclasses}
\begin{enumerate}
\item For $I \subset [n] = \{1,2,\cdots, n\}$ with $2 \le |I| \le n-2$, let $B_{I}$ be the closure of the locus of pointed curves $(C, x_{1}, \cdots, x_{n})$ with two irreducible components $C_{1}$ and $C_{2}$ such that $C_{1}$ (resp. $C_{2}$) contains $x_{i}$ for $i \in I$ (resp. $i \in I^{c}$). $B_{I}$ is called a boundary divisor. By the definition, $B_{I}= B_{I^{c}}$. For $2 \le i \le n-2$, let $B_{i} = \cup_{|I| = i}B_{I}$. Then $B_{i}$ is a symmetric divisor and $B_{i} = B_{n-i}$. Finally, let $B = \sum_{i=2}^{\lfloor n/2 \rfloor}B_{i}$. 
\item Fix $1 \le i \le n$. Let $\mathbb{L}_{i}$ be the line bundle on $\Mzn$ such that over $(C, x_{1}, \cdots, x_{n}) \in \Mzn$, the fiber is $\Omega_{C, x_{i}}$, the cotangent space of $C$ at $x_{i}$. Let $\psi_{i} = c_{1}(\mathbb{L}_{i})$, the $i$-th psi class. If we denote $\psi = \sum_{i=1}^{n}\psi_{i}$, then $\psi$ is a symmetric divisor. 
\item Let $K_{\Mzn}$ be the canonical divisor of $\Mzn$. Obviously it is symmetric.
\end{enumerate}
\end{definition}

The symmetric effective cone $\mathrm{Eff}(\Mzn)^{S_{n}} \cong \mathrm{Eff}(\Mzn/S_{n})$, which is $\mathrm{Eff}(\Mzn) \cap \mathrm{N}^{1}(\Mzn)_{\QQ}^{S_{n}}$, is generated by symmetric boundary divisors (\cite[Theorem 1.3]{KM96}). Therefore we can write $K_{\Mzn}$ and $\psi$ as nonnegative linear combinations of boundary divisors. 

\begin{lemma}\cite[Proposition 2]{Pan97}, \cite[Lemma 2.9]{Moo13}
On $\mathrm{N}^{1}(\Mzn)_{\QQ}$, the following relations hold.
\begin{enumerate}
	\item 
	$\displaystyle K_{\Mzn} = \sum_{i=2}^{\lfloor n/2\rfloor}
	\left(\frac{i(n-i)}{n-1} - 2\right)B_{i}$.
	\item $\displaystyle \psi = K_{\Mzn} + 2B$.
\end{enumerate}
\end{lemma}

\subsection{Curves on $\Mzn$}

Let $I_{1} \sqcup I_{2} \sqcup I_{3} \sqcup I_{4} = [n]$ be a partition. Let $F_{I_{1}, I_{2}, I_{3}, I_{4}}$ be the F-curve class corresponding to the partition (\cite[Section 4]{KM96}).

\begin{lemma}\cite{KM96}
Let $F = F_{I_{1}, I_{2}, I_{3}, I_{4}}$ be an F-curve and let $B_{J}$ be a boundary divisor. 
\begin{enumerate}
\item $F \cdot B_{J} = \begin{cases}1,&J = I_{i} \cup I_{j} \mbox{ for some } i \ne j,\\
-1, &J = I_{i} \mbox{ for some } i,\\
0, &\mbox{ otherwise.}\end{cases}$
\item $F \cdot \psi_{i} = \begin{cases}1, & I_{j} = \{i\} \mbox{ for some } j,\\
0, & \mbox{ otherwise.}\end{cases}$
\end{enumerate}
\end{lemma}

If we consider symmetric divisors only, then the intersection number does not depend on a specific partition but depends on the size of the partition. A curve class $F_{a_{1}, a_{2}, a_{3}, a_{4}}$ is one of any F-curve classes $F_{I_{1}, I_{2}, I_{3}, I_{4}}$ with $a_{i} = |I_{i}|$. 

To compute the stable base locus in Section \ref{sec:stablebaselocus}, we need to use other curve classes $C_{j}$ (see \cite[Lemma 4.8]{KM96}). Fix a $j$-pointed $\PP^{1}$ and let $x$ be an additional moving point on $\PP^{1}$. By gluing a fixed $(n-j+1)$-pointed $\PP^{1}$ whose last marked point is $y$ to the $(j+1)$-pointed $\PP^{1}$ along $x$ and $y$ and stabilizing it, we obtain an one parameter family of $n$-pointed stable curves over $\PP^{1}$, i.e., a curve $C_{j} \cong \PP^{1}$ on $\Mzn$.

\begin{lemma}\cite[Lemma 4.8]{KM96}
\[
	C_{j}\cdot B_{i} = \begin{cases}
	j, & i = j-1,\\ -(j-2), & i = j,\\ 0, & \mbox{otherwise.}
	\end{cases}
\]
\end{lemma}

\begin{remark}\label{rem:curveA}
We are able to generalize the idea of construction. For example, by 1) gluing two 3-pointed $\PP^{1}$ to $(n-2)$-pointed $\PP^{1}$, 2) varying one of two attached points, and 3) stabilizing it, we get an one parameter family of $n$-pointed stable curves over $\PP^{1}$. Let $A \subset \Mzv$ be such a curve class.
\end{remark}

\subsection{Numerical results on $\Mzv$}

For a convenience of readers, we leave a special case of $\Mzv$ below. All results are combinations of the Lemmas in previous sections. 

\begin{corollary}\label{cor:divisorclasscomputation}
The symmetric Neron-Severi space $\mathrm{N}^{1}(\Mzv)_{\QQ}^{S_{7}}$ has dimension two. The symmetric effective cone $\mathrm{Eff}(\Mzv)^{S_{7}}$ is generated by $B_{2}$ and $B_{3}$. Moreover, 
\begin{enumerate}
	\item $K_{\Mzv} = -\frac{1}{3}B_{2}$,
	\item $\psi = \frac{5}{3}B_{2}+2B_{3}$,
	\item $B_{2} = -3K_{\Mzv}$,
	\item $B_{3} = \frac{5}{2}K_{\Mzv}+\frac{1}{2}\psi$.
\end{enumerate}
\end{corollary}

We can summarize Corollary \ref{cor:divisorclasscomputation} with Figure \ref{fig:effectivecone}.

\begin{figure}[!ht]
\begin{tikzpicture}[scale=0.5]
	\draw[->][line width=1.5pt] (5, 0) -- (10, 0);
	\draw[line width=1.5pt] (0,0) -- (5,0);
	\draw[->][line width=1.5pt] (5,0) -- (5,5);
	\draw[->][line width=1.5pt] (5,0) -- (9.3,2.6);
	\draw[->][line width=1.5pt] (5,0) -- (0, 0);
	\node at (11,0) {$K_{\Mzv}$};
	\node at (5,5.5) {$\psi$};
	\node at (10, 2.6) {$B_{3}$};
	\node at (-0.5, 0) {$B_{2}$};
\end{tikzpicture}
\caption{Neron-Severi space of $\Mzv$}\label{fig:effectivecone}
\end{figure}
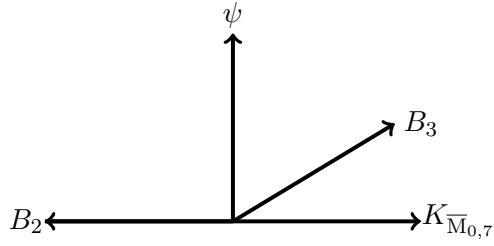

\begin{corollary}\label{cor:intersection}
On $\Mzv$, the intersection of symmetric divisors and curve classes are given by
Table \ref{tbl:intersectionnumbers}.
\begin{table}[!ht]
\begin{tabular}{|c|c|c|c|c|}
\hline
& $\psi$ & $K_{\Mzv}$ & $B_{2}$ & $B_{3}$\\ \hline
$F_{1,1,1,4}$ & 3 & -1 & 3 & -1\\ \hline
$F_{1,1,2,3}$ & 2 & 0 & 0 & 1\\ \hline
$F_{1,2,2,2}$ & 1 & 1 & -3 & 3 \\ \hline
$C_{4}$ & 4 & 0 & 0 & 2\\ \hline
$C_{5}$ & 5 & 1 & -3 & 5\\ \hline
$C_{6}$ & 10 & -2 & 6 & 0\\ \hline
$A$ & 3 & 1 & -3 & 4\\ \hline
\end{tabular}
\medskip
\caption{Intersection numbers on $\Mzv$}\label{tbl:intersectionnumbers}
\end{table}
\end{corollary}

\section{Stable base locus decomposition}\label{sec:stablebaselocus}

For an effective divisor $D$, the stable base locus $\mathbf{B}(D)$ is defined as
\[
	\mathbf{B}(D) = \bigcap_{m \ge 0}\mathrm{Bs}(mD),
\]
where $\mathrm{Bs}(D)$ is the set-theoretical base locus of $D$. As a first step toward Mori's program, we will compute stable base locus decompositions of $\Mzv$, which is  a first approximation of the chamber decompositions for different birational models. 

\begin{definition}
Let $B_{2}^{i}$ be the union of intersections of $i$ distinct irreducible components of $B_{2}$.
\end{definition}

Since $B$ is a simple normal crossing divisor, $B_{2}^{i}$ is a union of smooth varieties of codimension $i$. Moreover, the singular locus of $B_{2}^{i}$ is exactly $B_{2}^{i+1}$. On $\Mzv$, $B_{2}^{4}$ is an emptyset, $B_{2}^{3}$ is the union of all F-curves of type $F_{1,2,2,2}$. Each irreducible component of $B_{2}^{2}$ is isomorphic to $\Mzf$. Finally, $B_{2}^{1} = B_{2}$.

\begin{proposition}\label{prop:stablebaselocus}
Let $D$ be a symmetric effective divisor on $\Mzv$. Then:
\begin{enumerate}
	\item If $D \in [\psi-K_{\Mzv}, K_{\Mzv}+\frac{1}{3}\psi]$, 
	$D$ is semi-ample. 
	\item If $D \in (K_{\Mzv} + \frac{1}{3}\psi, B_{3}]$, 
	$\mathbf{B}(D) = B_{3}$
	\item If $D \in [\psi-3K_{\Mzv},\psi -K_{\Mzv})$, 
	$\mathbf{B}(D) = B_{2}^{3}$. 
	\item If $D \in [\psi-5K_{\Mzv}, \psi-3K_{\Mzv})$, 
	$\mathbf{B}(D) = B_{2}^{2}$. 
	\item If $D \in [B_{2}, \psi-5K_{\Mzv})$, 
	$\mathbf{B}(D) = B_{2}$. 
\end{enumerate}
\end{proposition}

\begin{proof}
By \cite[Theorem 1.2]{KM96} and Corollary \ref{cor:intersection}, the nef cone of $\Mzv$ is generated by $\psi-K_{\Mzv}$ and $K_{\Mzv}+\frac{1}{3}\psi$. Moreover, $K_{\Mzv}+\frac{1}{3}\psi$ is the pull-back of an ample divisor on $\Mza$ where $A = (\frac{1}{3}, \frac{1}{3}, \cdots, \frac{1}{3})$ (See the proof of Theorem 3.1 of \cite{Moo13}). In particular, the right hand side of Equation (7) is zero.). The opposite extremal ray $\psi-K_{\Mzv}$ is also semi-ample. Indeed, by comparing the intersection numbers, it is straightforward that $\psi-K_{\Mzv}$ is proportional to the pull-back of the canonical polarization on the Veronese quotient $V_{A}^{3}$ where $A = (\frac{3}{7}, \cdots, \frac{3}{7})$ (\cite[Theorem 2.1]{GJMS13}). Therefore two endpoints of this interval, and hence all divisors in the interval are semi-ample divisors. 

If $D \in (K_{\Mzv}+\frac{1}{3}\psi, B_{3}]$, then $\mathbf{B}(D) \subset B_{3}$ since $K_{\Mzv}+\frac{1}{3}\psi$ is semi-ample and $D$ is an effective linear combination of $K_{\Mzv}+\frac{1}{3}\psi$ and $B_{3}$. By Corollary \ref{cor:intersection}, $F_{1,1,1,4}\cdot D < 0$ so $F_{1,1,1,4} \subset \mathbf{B}(D)$. Since $F_{1,1,1,4}$ covers an open dense subset of $B_{3}$, $\mathbf{B}(D) = B_{3}$. 

If $D \in [B_{2}, \psi-K_{\Mzv})$, then $\mathbf{B}(D) \subset B_{2}$ by a similar reason. By Corollary \ref{cor:intersection}, $F_{1,2,2,2}\cdot D < 0$ if $D \in [B_{2}, \psi-K_{\Mzv})$, thus $F_{1,2,2,2} \subset \mathbf{B}(D)$. If $D \in [B_{2}, \psi-3K_{\Mzv})$, $A \cdot D < 0$ and $A$ covers a dense open subset of $B_{2}^{2}$. Thus $B_{2}^{2} \subset \mathbf{B}(D)$. Finally, if $D \in [B_{2}, \psi-5K_{\Mzv})$, $C_{5}\cdot D < 0$. Since $C_{5}$ covers an open dense subset of $B_{2}$, $B_{2} \subset \mathbf{B}(D)$. In particular, we obtain Item (5).

Now it is sufficient to show that $\mathbf{B}(D) \subset B_{2}^{3}$ if $D \in [\psi-3K_{\Mzv}, \psi-K_{\Mzv})$ and $\mathbf{B}(D) \subset B_{2}^{2}$ if $[\psi-5K_{\Mzv}, \psi-3K_{\Mzv})$. Let $B_{I}$ be an irreducible component of $B_{2}$ and $B_{J}$ be an irreducible component of $B_{3}$ such that $B_{I} \cap B_{J} \ne \emptyset$. For $E = 5B_{2}+3B_{3} = \frac{3}{2}(\psi-5K_{\Mzv})$, by using Keel's relations (\cite[550p]{Kee92}) and a computer algebra system, we can find a divisor $E' \in |E|$ such that $E'$ is a non-negative integral linear combination of boundary divisors such that the coefficients of $B_{I}$ and $B_{J}$ are zero. For example, if $I = \{1,2\}$ and $J = \{3,4,5\}$, 
\begin{eqnarray*}
E &\equiv& 12B_{\{1,4\}}+9\left(B_{\{2,5\}}+B_{\{2,6\}}+B_{\{5,6\}}\right)\\
&& +6\left(B_{\{1,3\}}+B_{\{1,7\}}+B_{\{2,3\}}+B_{\{2,7\}}+B_{\{3,4\}}+B_{\{3,7\}}+B_{\{4,7\}}\right)\\
&&+3\left(B_{\{1,5\}}+B_{\{1,6\}}+B_{\{3,5\}}+B_{\{3,6\}}+B_{\{4,5\}}+B_{\{4,6\}}+B_{\{5,7\}}+B_{\{6,7\}}\right)\\
&& + 15B_{\{2,5,6\}}+12\left(B_{\{1,4,7\}}+B_{\{1,3,4\}}\right)\\
&& + 6\left(B_{\{1,3,7\}}+B_{\{1,4,5\}}+B_{\{1,4,6\}}+B_{\{2,3,5\}} + B_{\{2,3,6\}}+B_{\{2,3,7\}}+B_{\{2,5,7\}}+B_{\{2,6,7\}}+B_{\{3,4,7\}}\right)\\
&& + 3\left(B_{\{1,5,6\}}+B_{\{3,5,6\}}+B_{\{4,5,6\}}+B_{\{5,6,7\}}\right).
\end{eqnarray*}
Similarly, if $I = \{1,2\}$ and $J = \{1,2,3\}$, 
\begin{eqnarray*}
E &\equiv& 12B_{\{1,4\}} + 9\left(B_{\{2,6\}}+B_{\{2,7\}}+B_{\{6,7\}}\right)\\
&& + 6\left(B_{\{1,3\}}+B_{\{1,5\}}+B_{\{2,3\}}+B_{\{2,5\}}+B_{\{3,4\}}+B_{\{3,5\}}+B_{\{4,5\}}\right)\\
&& + 3\left(B_{\{1,6\}}+B_{\{1,7\}}+B_{\{3,6\}}+B_{\{3,7\}}+B_{\{4,6\}}+B_{\{4,7\}}+B_{\{5,6\}}+B_{\{5,7\}}\right)\\
&& + 15B_{\{2,6,7\}} + 12\left(B_{\{1,3,4\}}+B_{\{1,4,5\}}\right)\\
&& + 6\left(B_{\{1,3,5\}}+B_{\{1,4,6\}}+B_{\{1,4,7\}}+B_{\{2,3,5\}}+B_{\{2,3,6\}}+B_{\{2,3,7\}}+B_{\{2,5,6\}}+B_{\{2,5,7\}}+B_{\{3,4,5\}}\right)\\
&& + 3\left(B_{\{1,6,7\}}+B_{\{3,6,7\}}+B_{\{4,6,7\}}+B_{\{5,6,7\}}\right).
\end{eqnarray*}
These two cases cover all cases that $B_{I} \cap B_{J} \ne \emptyset$ up to the $S_{7}$-action. Thus the support of $E'$ does not contain a general point of $B_{I}$ and a general point of $B_{I} \cap B_{J}$. Therefore $\mathbf{B}(E)$ must be contained in $B_{2}^{2}$. Since $\psi-K_{\Mzv}$ is semi-ample, for all divisor $D \in [\psi-5K_{\Mzv}, \psi-K_{\Mzv})$, $\mathbf{B}(D) \subset B_{2}^{2}$ and Item (4) was shown.

Finally, let $B_{I}, B_{K}$ be two irreducible components of $B_{2}$ whose intersection is nonempty. For $F = 4B_{2}+3B_{3}= \frac{3}{2}(\psi-3K_{\Mzv})$, by using a similar idea, we can find a divisor $F' \in |F|$ such that $F'$ is a non-negative integral linear combination of boundary divisors such that the coefficients of $B_{I}$ and $B_{K}$ are zero. Indeed, if $I = \{1,2\}$ and $K = \{3,4\}$, 
\begin{eqnarray*}
F &\equiv& 12B_{\{1,3\}}+9\left(B_{\{2,4\}}+B_{\{2,6\}}+B_{\{4,6\}}\right)\\
&& + 6\left(B_{\{1,5\}}+B_{\{1,7\}}+B_{\{3,5\}}+B_{\{3,7\}}\right)\\
&& + 3\left(B_{\{2,5\}}+B_{\{2,7\}}+B_{\{4,5\}}+B_{\{4,7\}}+B_{\{5,6\}}+B_{\{5,7\}}+B_{\{6,7\}}\right)\\
&& +18B_{\{2,4,6\}} + 15\left(B_{\{1,3,5\}}+B_{\{1,3,7\}}\right)\\
&& + 6\left(B_{\{1,5,7\}}+B_{\{2,4,5\}}+B_{\{2,4,7\}}+B_{\{2,5,6\}}+B_{\{2,6,7\}}+B_{\{3,5,7\}}+B_{\{4,5,6\}}+B_{\{4,6,7\}}\right)\\
&& + 3\left(B_{\{1,2,3\}}+B_{\{1,3,4\}}+B_{\{1,3,6\}}\right).
\end{eqnarray*}
Thus a general point of $B_{2}^{2}$ is not contained in $\mathbf{B}(F)$, too. The only remaining locus in $B_{2}$ is $B_{2}^{3}$. Hence $\mathbf{B}(F) \subset B_{2}^{3}$ and the same holds for all $D \in [\psi-3K_{\Mzv}, \psi-K_{\Mzv})$. 
\end{proof}

We summarize the above result as Figure \ref{fig:stablebaselocus}. 

\begin{figure}[!ht]
\begin{tikzpicture}[scale=0.8]
	\draw[->][line width=1.5pt][lightgray] (5, 0) -- (10, 0);
	\draw[<-][line width=1.5pt] (0,0) -- (5,0);
	\draw[->][line width=1.5pt][lightgray] (5,0) -- (5,5);
	\draw[->][line width=1.5pt] (5,0) -- (9.3, 2.6);
	\draw[->][line width=1.5pt] (5,0) -- (8.5, 3.5);
	\draw[->][line width=1.5pt] (5,0) -- (3.4, 4.7);
	\draw[->][line width=1.5pt] (5,0) -- (1.5, 3.5);
	\draw[->][line width=1.5pt] (5,0) -- (0.7, 2.6);
	\node at (11,0) {$K_{\Mzv}$};
	\node at (-1, 0) {$B_{2}$};
	\node at (5,5.5) {$\psi$};
	\node at (9.8, 2.6) {$B_{3}$};
	\node at (8.7, 4) {$K_{\Mzv} + \frac{1}{3}\psi$};
	\node at (8.2, 2.5) {$B_{3}$};
	\node at (1, 1) {$B_{2}$};
	\node at (5.5,3.5) {$\emptyset$};
	\node at (1.9,2.4) {$B_{2}^{2}$};
	\node at (3,3.2) {$B_{2}^{3}$};
	\node at (3.4, 5) {$\psi - K_{\Mzv}$};
	\node at (1.5, 3.8) {$\psi - 3K_{\Mzv}$};
	\node at (-0.5, 2.6) {$\psi - 5K_{\Mzv}$};
\end{tikzpicture}
\caption{Stable base locus decomposition of $\Mzv$}\label{fig:stablebaselocus}
\end{figure}
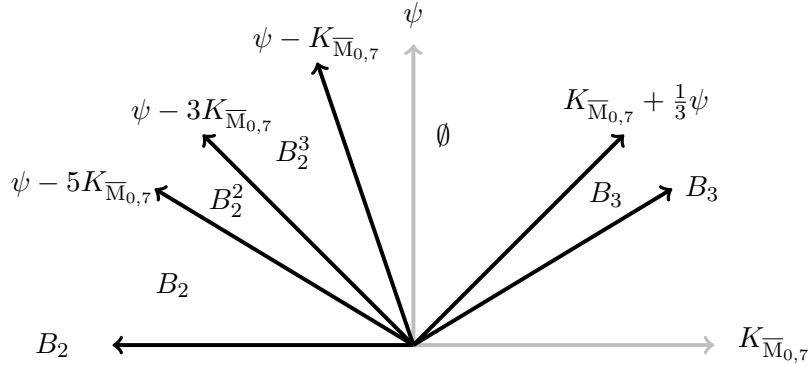

\section{Mori's program for $\Mzv$}\label{sec:Moriprogram}

In this section, we show the first main theorem (Theorem \ref{thm:Moriprogramintro}) of this paper. 

\begin{theorem}\label{thm:Moriprogram}
Let $D$ be a symmetric effective divisor of $\Mzv$. Then:
\begin{enumerate}
\item If $D \in (\psi - K_{\Mzv}, K_{\Mzv} + \frac{1}{3}\psi)$, $\Mzv(D) \cong \Mzv$. 
\item If $D \in [K_{\Mzv}+\frac{1}{3}\psi, B_{3})$, $\Mzv(D) \cong \Mza$, the moduli space of weighted pointed stable curves with weight $A = \left(\frac{1}{3}, \cdots, \frac{1}{3}\right)$.
\item If $D = \psi - K_{\Mzv}$, $\Mzv(D)$ is isomorphic to the Veronese quotient $V_{A}^{3}$ where $A = \left(\frac{4}{7}, \cdots, \frac{4}{7}\right)$.
\item If $D \in (\psi - 3K_{\Mzv}, \psi - K_{\Mzv})$, $\Mzv(D) \cong \Mzv^{3}$, which is a flip of $\Mzv$ over $V_{A}^{3}$. The flipping locus is $B_{2}^{3}$.
\item If $D = \psi - 3K_{\Mzv}$, $\Mzv(D)$ is a small contraction of $\Mzv^{3}$.
\item If $D \in (\psi - 5K_{\Mzv}, \psi - 3K_{\Mzv})$, $\Mzv(D) \cong \Mzv^{2}$, which is a flip of $\Mzv^{3}$ over $\Mzv(\psi - 3K_{\Mzv})$. The flipping locus is the proper transform of $B_{2}^{2}$.
\item If $D \in (B_{2}, \psi - 5K_{\Mzv}]$, $\Mzv(D) \cong \Mzv^{1}$, which is a divisorial contraction of $\Mzv^{2}$. The contracted divisor is the proper transform of $B_{2}$.
\item If $D = B_{2}$ or $B_{3}$, $\Mzv(D)$ is a point.
\end{enumerate}
\end{theorem}

Before proving Theorem \ref{thm:Moriprogram}, we describe some moduli spaces appear on the theorem. 

\subsection{Moduli of weighted pointed stable curves}\label{ssec:Hassettspace}
The moduli space $\Mza$ of weighted pointed stable curves, in Item (2), is constructed in \cite{Has03}. For a collection of positive rational numbers (so called weight data) $A = (a_{1}, a_{2}, \cdots, a_{n})$ with $0 < a_{i} \le 1$ and $\sum a_{i} > 2$, there is a fine moduli space of pointed curves $(C, x_{1}, \cdots, x_{n})$ such that
\begin{itemize}
\item $C$ is a reduced, connected projective curve of $p_{a}(C) = 0$;
\item $(C, \sum a_{i}x_{i})$ is a semi-log canonical pair;
\item $\omega_{C}+\sum a_{i}x_{i}$ is ample.
\end{itemize}
In contrast to $\Mzn$, for a subset $I \subset [n]$, if $\sum_{i \in I}a_{i}\le 1$ then $\{x_{i}\}_{i \in I}$ may collide at a smooth point of $C$. But because of the last condition, each tail of $C$ has sufficiently many marked points in the sense that their weight sum is greater than one. Also note that $\Mzn = \overline{\mathrm{M}}_{0, (1,1,\cdots, 1)}$. 

The moduli space $\Mza$ is smooth and birational to $\Mzn$. Furthermore, there is a reduction map $\rho_{A} : \Mzn \to \Mza$ for any weight data, which is a divisorial contraction. The map $\rho_{A}$ sends a pointed curve $(C, x_{1}, x_{2}, \cdots, x_{n})$ to a new curve $(\overline{C}, \bar{x}_{1}, \bar{x}_{2}, \cdots, \bar{x}_{n})$ which is obtained by contracting all tails with weight sums $\le 1$ to the attaching point. 

\begin{example}
For the case of $n = 7$ and $A = \left(\frac{1}{3}, \cdots, \frac{1}{3}\right)$, $\rho_{A}$ is the contraction of $B_{3}$. A general point $(C_{1} \cup C_{2}, x_{1}, x_{2}, \dots, x_{7})$ has a tail with three marked points. Then the sum is precisely one, so the tail is contracted to a point. Note that it forgets the cross ratio of three marked points and a nodal point. Thus the image of $B_{3}$ is a codimension two subvariety of $\Mza$. Figure \ref{fig:reductionmap} shows the contraction. The number on a marked point is the multiplicity. 

\begin{figure}[!ht]
\begin{tikzpicture}[scale=0.4]
	\draw[line width=1pt] (0, 5) -- (6, 0);
	\draw[line width=1pt] (5, 0) -- (12, 5);
	\fill (1.2, 4) circle (5pt);
	\fill (2.4, 3) circle (5pt);
	\fill (3.6, 2) circle (5pt);
	\fill (7.4, 1.8) circle (5pt);
	\fill (8.4, 2.5) circle (5pt);
	\fill (9.4, 3.2) circle (5pt);
	\fill (10.4, 3.9) circle (5pt);

	\node (to) at (15, 3) {$\Rightarrow$};
	
	\draw[line width=1pt] (17, 0) -- (24, 5);
	\fill (19.4, 1.8) circle (5pt);
	\fill (20.4, 2.5) circle (5pt);
	\fill (21.4, 3.2) circle (5pt);
	\fill (22.4, 3.9) circle (5pt);
	\fill (17.6, 0.5) circle (5pt);
	\node (3) at (17.6, 1.3) {$3$};
\end{tikzpicture}
\caption{The reduction map $\rho_{A} : \Mzv \to \Mza$ where $A = \left(\frac{1}{3}, \cdots, \frac{1}{3}\right)$}\label{fig:reductionmap}
\end{figure}
\end{example}

\subsection{Veronese quotients}\label{ssec:Veronesequotients}

The Veronese quotients $V_{A}^{d}$ in Item (3) and their geometric properties have been studied in \cite{Gia13, GJM13, GJMS13}. Originally, they are constructed as GIT quotients of an incidence variety of the Chow variety of rational normal curves in $\PP^{d}$ and projective spaces. 

Let $\mathrm{Chow}_{1,d}(\PP^{d})$ be the irreducible component of the Chow variety which parametrizes rational normal curves and their degenerations. Consider the incidence variety 
\[
	I := \{(C, x_{1}, \cdots, x_{n}) \in \mathrm{Chow}_{1,d}(\PP^{d})
	 \times (\PP^{d})^{n} \;|\; x_{i} \in C\}.
\]
There is a natural $\SL_{d+1}$-action on $I$ and $\mathrm{Chow}_{1,d}(\PP^{d}) \times (\PP^{d})^{n}$. Also there is a canonical polarization $\cO_{\mathrm{Chow}}(1)$ on $\mathrm{Chow}_{1,d}(\PP^{d})$. For a sequence of nonnegative rational numbers $(\gamma, a_{1}, a_{2}, \cdots, a_{n})$, define a $\QQ$-polarization on $I$ which is the pull-back of 
\[
	L_{A} := \cO_{\mathrm{Chow}}(\gamma) \otimes \cO(a_{1}) \otimes \cdots 
	\otimes \cO(a_{n})
\]
on $\mathrm{Chow}_{1,d}(\PP^{d}) \times (\PP^{d})^{n}$. We will normalize the linearization by imposing a numerical condition $(d-1)\gamma + \sum a_{i} = d+1$. Thus $\gamma$ is determined by $A := (a_{1}, a_{2}, \cdots, a_{n})$ and $d$. If $0 < a_{i} < 1$ and $2 < \sum a_{i} \le d+1$ (hence $0 \le \gamma < 1$), then the semistable locus $I^{ss}$ is nonempty (\cite[Proposition 2.10]{GJM13}), so we are able to obtain a nonempty GIT quotient $V_{A}^{d} := I\git_{L_{A}}\SL_{d+1}$. 

\begin{remark}
A simple observation on the semistability is that every stable curve is non-degenerate. A non-degenerate degree $d$ curve in $\PP^{d}$ has several nice geometric properties: 1) Every connected subcurve of degree $e$ spans $\PP^{e} \subset \PP^{d}$, and 2) all singularities are analytically locally the union of coordinate axes in some $\CC^{k}$ (\cite[Corollary 2.4]{GJM13}). 
\end{remark}

For simplicity, consider general polarizations such that $I^{ss} = I^{s}$. These quotients have modular interpretation, as moduli spaces of stable polarized pointed curves. For a precise definition and proof, consult \cite[Section 5.1]{GJM13}. 

For any weight data $A$ and $d > 0$, there is a reduction map $\phi : \Mzn \to V_{A}^{d}$ (\cite[Theorem 1.1]{GJM13}), which preserves $\mathrm{M}_{0,n}$. For each (possibly reducible) connected \emph{tail} $C'$ of $(C, x_{1}, x_{2}, \cdots, x_{n}) \in \Mzn$, we may define a numerical value 
\[
	\sigma(C') := 
	\min\left\{\max\left\{\Big\lceil \frac{\sum_{x_{i} \in C'}a_{i} - 1}{1-\gamma}\Big\rceil,
	0\right\}, d\right\}.
\]
Because the dual graph of $C$ is a tree, we can define $\sigma(C')$ for \emph{every} irreducible component $C'$, by setting that $\sigma(C') := \sigma(C'' \cup C') - \sigma(C'')$ for any tail $C''$ such that $C'' \cup C'$ is connected. The reduction map $\phi$ sends $(C, x_{1}, x_{2}, \cdots, x_{n})$ to a new curve $(\overline{C}, \bar{x}_{1}, \cdots, \bar{x}_{n})$ which is obtained by contracting all irreducible components $C'$ with $\sigma(C') = 0$. 

\begin{example}\label{ex:Veronesequotient}
Consider $n = 7$, $d = 3$ and $A = \left(\frac{4}{7}, \cdots, \frac{4}{7}\right)$ (hence $\gamma = 0$) case. Then there are only two types of curves in $\Mzv$ with contractions.
\begin{enumerate}
\item A chain of curves $C = C_{1} \cup C_{2} \cup C_{3}$ such that $C_{1}$ with two marked points, $C_{2}$ with a marked point, and (possibly reducible) $C_{3}$ with four marked points. Then $C_{2}$ is contracted to a point. 
\item A comb of rational curves with three tails $C_{1}, C_{2}, C_{3}$ with two marked points respectively, and a spine $C_{4}$ with a marked points. $C_{4}$ is contracted to a triplenodal singularity with a marked point on it.
\end{enumerate}
Note that for the first case, the contracted component has only three special points. Thus around the point, $\Mzv$ and $V_{A}^{3}$ are locally isomorphic. But in the second case, the spine has four special points so it has a one-dimensional moduli. Thus the map $\phi$ contracts the loci of such curves, which are F-curves of type $F_{1,2,2,2}$. So $\phi$ is a small contraction.
\end{example}

\begin{figure}[!ht]
\begin{tikzpicture}[scale=0.5]
	\draw[line width=1pt] (0, 1) -- (6, 0);
	\draw[line width=1pt] (0, 4) -- (6, 5);
	\draw[line width=1pt] (1, 0) -- (1, 5);
	\fill (1, 2.5) circle (4pt);
	\fill (2, 4.35) circle (4pt);
	\fill (4, 4.65) circle (4pt);
	\fill (2, 0.65) circle (4pt);
	\fill (3, 0.5) circle (4pt);
	\fill (4, 0.3) circle (4pt);
	\fill (5, 0.15) circle (4pt);
	\node (to1) at (8, 2.5) {$\Rightarrow$};
	\draw[line width=1pt] (10, 2) -- (16, 5);
	\draw[line width=1pt] (10, 3) -- (16, 0);
	\fill (11, 2.5) circle (4pt);
	\fill (13, 3.5) circle (4pt);
	\fill (15, 4.5) circle (4pt);
	\fill (12, 2) circle (4pt);
	\fill (13, 1.5) circle (4pt);
	\fill (14, 1) circle (4pt);
	\fill (15, 0.5) circle (4pt);
	
	\draw[line width=1pt] (0, -5) -- (6, -6);
	\draw[line width=1pt] (0, -2) -- (6, -1);
	\draw[line width=1pt] (0, -3.5) -- (6, -3.5);
	\draw[line width=1pt] (1, -6) -- (1, -1);
	\fill (1, -2.7) circle (4pt);
	\fill (2, -1.65) circle (4pt);
	\fill (4, -1.35) circle (4pt);
	\fill (2, -5.35) circle (4pt);
	\fill (4, -5.7) circle (4pt);
	\fill (2, -3.5) circle (4pt);
	\fill (4, -3.5) circle (4pt);
	\node (to2) at (8, -3.5){$\Rightarrow$};
	\draw[line width=1pt] (13, -5.5) -- (13, -1);
	\draw[line width=1pt] (14, -4) -- (10, -6);
	\draw[line width=1pt] (12, -4) -- (16, -6);
	\fill (13, -4.5) circle (4pt);
	\fill (12, -5) circle (4pt);
	\fill (11, -5.5) circle (4pt);
	\fill (14, -5) circle (4pt);
	\fill (15, -5.5) circle (4pt);
	\fill (13, -3) circle (4pt);
	\fill (13, -2) circle (4pt);		
\end{tikzpicture}
\caption{The reduction map $\phi : \Mzv \to V_{A}^{3}$ where $A = \left(\frac{4}{7}, \cdots, \frac{4}{7}\right)$}\label{fig:reductionmapVeronese}
\end{figure}

\begin{remark}\label{rem:GITquotient}
An important observation for Example \ref{ex:Veronesequotient} is that we may replace the Chow variety by moduli space of stable maps $\Mzzptt$. There is a cycle map
\[
	f : \overline{\mathrm{M}}_{0,0}(\PP^{d}, d) \to \mathrm{Chow}_{1,d}(\PP^{d}).
\]
When $d \le 3$, If we take the locus $\overline{\mathrm{M}}_{0,0}(\PP^{d}, d)^{nd}$ parametrizes stable maps with non-degenerated images and if $\mathrm{Chow}_{1,d}(\PP^{d})^{nd}$ is the image of it, then the restricted cycle map is isomorphism because there is no degree 0 component with positive dimensional moduli. Therefore 
\[
	\overline{\mathrm{M}}_{0,0}(\PP^{3}, 3)^{nd} \times (\PP^{3})^{n} 
	\to \mathrm{Chow}_{1,3}(\PP^{3})^{nd} \times (\PP^{3})^{n}
\]
is an isomorphism and $I^{s}$ is a subset of $\mathrm{Chow}_{1,3}(\PP^{3})^{nd} \times (\PP^{3})^{n}$. Therefore we may replace the Chow variety by $\Mzzptt$. 

Furthermore, $\Mzn \cong \Mznpdd\git \SL_{d+1}$ for an appropriate linearization (\cite[Proposition 4.6]{GJM13}). And the morphism $\Mzv \to V_{A}^{3}$ is obtained by taking quotient map of 
\[
	\Mzvptt \to \Mzzptt \times (\PP^{3})^{7}.
\]
\end{remark}

The other birational models $\Mzv^{i}$ with $i = 1,2,3$ are new spaces which don't appear on literatures. We will describe them concretely using explicit blow-ups and downs.

\subsection{Outline of the proof}\label{ssec:outline}

The proof of Theorem \ref{thm:Moriprogram} involves explicit but long computations of several birational modifications. So we leave an outline of the proof here and prove it in next several sections. 

\begin{proof}[Outline of the proof of Theorem \ref{thm:Moriprogram}]
Since the symmetric nef cone is generated by $\psi - K_{\Mzv}$ and $K_{\Mzv}+\frac{1}{3}\psi$, $D$ in Item (1) is an ample divisor. Thus $\Mzv(D) \cong \Mzv$. 

Item (2) is established in \cite[Theorem 3.1]{Moo13}. If $D = K_{\Mzv} + \frac{1}{3}\psi$, $\Mzv(D) \cong \Mza$. Because for $D$ in the range of Item (2) the stable base locus $\mathbf{B}(D)$ is $B_{3}$, after removing $B_{3}$, we obtain Item (2) in general.

Consider the reduction map $\phi : \Mzv \to V_{A}^{3}$ in Item (3). By applying \cite[Theorem 3.1]{GJMS13}, we can compute the pull-back $D_{A}$ of the canonical polarization on $V_{A}^{3}$. With the notation in \cite{GJMS13}, Item (3) is the case that $\gamma = 0$, $A = \left(\frac{4}{7}, \frac{4}{7}, \cdots, \frac{4}{7}\right)$. So it is straightforward to check that $F_{1,2,2,2}\cdot D_{A} = 0$. Since $\dim \mathrm{N}^{1}(\Mzv)^{S_{7}}_{\QQ} = 2$, this implies that $D_{A}$ is proportional to $\psi - K_{\Mzv}$ by Corollary \ref{cor:intersection}. Therefore $\Mzv(\psi - K_{\Mzv}) \cong \Mzv(D_{A}) \cong V_{A}^{3}$. 

Items (4), (5), (6), and (7) are obtained by careful computations of flips and contractions. We give a proof of Item (4) in Proposition \ref{prop:firstflip}. Items (5) and (6) are proved in Lemma \ref{lem:flippingbase} and Proposition \ref{prop:secondflip} respectively. We prove Item (7) in Proposition \ref{prop:divcontraction}. 

Since $B_{2}$ and $B_{3}$ are rigid, Item (8) follows immediately.
\end{proof}

\begin{remark}
The direction toward canonical divisor have been well understood for all $n$ and all (possibly non-symmetric) weight data. For every $n$ and $A = \left(a_{1}, a_{2}, \cdots, a_{n}\right)$, 
\[
	\Mzn(K_{\Mzn}+\sum a_{i}\psi_{i}) \cong \Mza.
\]
For a proof, see \cite{Moo13}. Also for a generalization to $\overline{\cM}_{g,n}$ with $g > 0$, consult \cite{Moo11b}.
\end{remark}

\subsection{First flip} \label{ssec:firstflip}

In this section, we describe the first flip $\Mzv \dashrightarrow \Mzv^{3}$ in terms of blow-ups and downs. 

\begin{proposition}\label{prop:firstflip}
Let $\widetilde{\mathrm{M}}_{0,7}^{3}$ be the blow-up of $\Mzv$ along $B_{2}^{3}$. A connected component of the exceptional locus is isomorphic to $\PP^{1} \times \PP^{2}$. Let $\Mzv^{3}$ be the blow-down of these exceptional locus to the opposite direction. Then $\Mzv^{3}$ is smooth and it is $D$-flip of $\phi : \Mzv \to V_{A}^{3}$ for $D \in (\psi - 3K_{\Mzv}, \psi - K_{\Mzv})$ and $\Mzv(D) \cong \Mzv^{3}$.
\end{proposition}

\begin{proof}
On $\Mzv$, $B_{2}^{3}$ is the disjoint union of 105 F-curves of type $F_{1,2,2,2}$. Take a component $F$ of $B_{2}^{3}$, which is an F-curve $B_{I} \cap B_{J} \cap B_{K}$ where $|I| = |J| = |K| = 2$. The normal bundle $N := N_{F/\Mzv}$ is isomorphic to $\cO(B_{I}) \oplus \cO(B_{J}) \oplus \cO(B_{K})|_{F}$. By \cite[Lemma 4.5]{KM96}, $N \cong \cO(-\psi_{p}) \oplus \cO(-\psi_{q}) \oplus \cO(-\psi_{r})$ where $p$, $q$, $r$ are attaching points of three tails. Since $F \cdot \psi_{x} = 1$ for any attaching point $x$, $N \cong \cO_{\PP^{1}}(-1)^{3}$. 

Let $\pi_{3}: \widetilde{\mathrm{M}}_{0,7}^{3} \to \Mzv$ be the blow-up. The blown-up space $\widetilde{\mathrm{M}}_{0,7}^{3}$ is a smooth variety. Also a connected component $E$ of the exceptional locus is $\PP(N) \cong \PP(\cO_{\PP^{1}}(-1)^{3}) \cong \PP^{1} \times \PP^{2}$ and the normal bundle $N_{E/\widetilde{\mathrm{M}}_{0,7}^{3}}$ is isomorphic to $\cO_{\PP^{1}\times \PP^{2}}(-1,-1)$. Thus for a point $y \in \PP^{2}$, the restricted normal bundle to a fiber $\PP^{1} \times \{y\}$ is $\cO_{\PP^{1}}(-1)$. Therefore there exists a smooth contraction $\Mzv^{3}$, which contracts the $\PP^{1}$-fibration structure of the exceptional divisor. Let $\pi_{3}': \widetilde{\mathrm{M}}_{0,7}^{3} \to \Mzv^{3}$ be the contraction. Since the positive dimensional fiber of $\pi_{3}'$ is contracted by $\phi \circ \pi_{3}$, there is a birational map $\phi_{3}' : \Mzv^{3} \to V_{A}^{3}$ such that $\phi \circ \pi_{3} = \phi_{3}' \circ \pi_{3}'$ by rigidity lemma (\cite[Proposition II.5.3]{Kol96}).
\[
	\xymatrix{& \widetilde{\mathrm{M}}_{0,7}^{3}
	\ar_{\pi_{3}}[ld]\ar^{\pi_{3}'}[rd]\\
	\Mzv\ar_{\phi}[rd] && \Mzv^{3}\ar^{\phi_{3}'}[ld]\\ & V_{A}^{3}}
\]

We claim that $\phi_{3}': \Mzv^{3} \to V_{A}^{3}$ is a $D$-flip for $D \in (\psi-3K_{\Mzv}, \psi-K_{\Mzv})$. The exceptional set of $\phi$ is exactly $B_{2}^{3} = \cup F_{1,2,2,2}$. From Corollary \ref{cor:intersection}, $-D \cdot F_{1,2,2,2} > 0$. Thus $-D$ is $\phi$-ample. Note that a connected component of the positive dimensional exceptional locus of $\phi_{3}'$ is isomorphic to $\PP^{2}$. Let $\widetilde{L}$ be a line class of type $(0,1)$ in the exceptional divisor $E \cong \PP^{1} \times \PP^{2}$ on $\widetilde{\mathrm{M}}_{0,7}^{3}$. And let $L := \pi_{3}'(\widetilde{L})$ which is a line on the exceptional locus of $\phi_{3}'$. Note that on $\phi_{3}'$-exceptional $\PP^{2}$, $B_{I}|_{\PP^{2}}$, $B_{J}|_{\PP^{2}}$, $B_{K}|_{\PP^{2}}$ are line classes. So $B_{2} \cdot L = 3$. On the other hand, $B_{3}$ intersects $E$ three times and each irreducible component of the intersection is isomorphic to $\{*\} \times \PP^{2} \subset \PP^{1} \times \PP^{2} \cong E$, the divisor $B_{3}$ on $\Mzv^{3}$ vanishes along $\PP^{2}$ with multiplicity three. Hence $B_{3}\cdot L = -3$. Now from $\psi - K_{\Mzv} = 2B_{2}+2B_{3}$, for $D \in (B_{2}, \psi - K_{\Mzv})$, $D \cdot L > 0$ so $D$ is $\phi_{3}'$-ample.

Furthermore, we can see that for $D \in (\psi-3K_{\Mzv}, \psi-K_{\Mzv})$, $D$ is ample on $\Mzv^{3}$. If a curve class $C$ is in the image of exceptional $\PP^{2}$, then we already proved that $C \cdot D \ge 0$. If $C$ is not contained in the exceptional locus, from Proposition \ref{prop:stablebaselocus}, $mD$ is movable for $m \gg 0$ on the outside of $B_{2}^{3}$ thus $C \cdot D \ge 0$ if $D \in [\psi - 3K_{\Mzv}, \psi-K_{\Mzv}]$. Therefore the nef cone of $\Mzv^{3}/S_{7}$ is generated by $\psi - K_{\Mzv}$ and $\psi-3K_{\Mzv}$. Since the ample cone is the interior of the nef cone, the desired result follows. 
\end{proof}

\begin{remark}
After the first flip, the proper transform of $B_{2}^{2}$ becomes a disjoint union of its irreducible components. Each irreducible component is isomorphic to $\PP^{1} \times \PP^{1}$. 
\end{remark}

\subsection{Second flip}\label{ssec:secondflip}

The description of the second flip is more complicate. It is a composition of two smooth blow-ups, a smooth blow-down and a singular blow-down. In this section, we will describe the second flip. Since the flipping locus is the disjoint union of irreducible components of the proper transform of $B_{2}^{2}$, it is enough to focus on the modification on an irreducible component. We will give an outline of the description first, and after that we give justifications of statements as a collection of lemmas. Figure \ref{fig:secondflip} shows the decomposition of the flip. By abusing notation, we say $B_{2}^{2}$ for the proper transform of $B_{2}^{2}$ on $\Mzv^{3}$. 

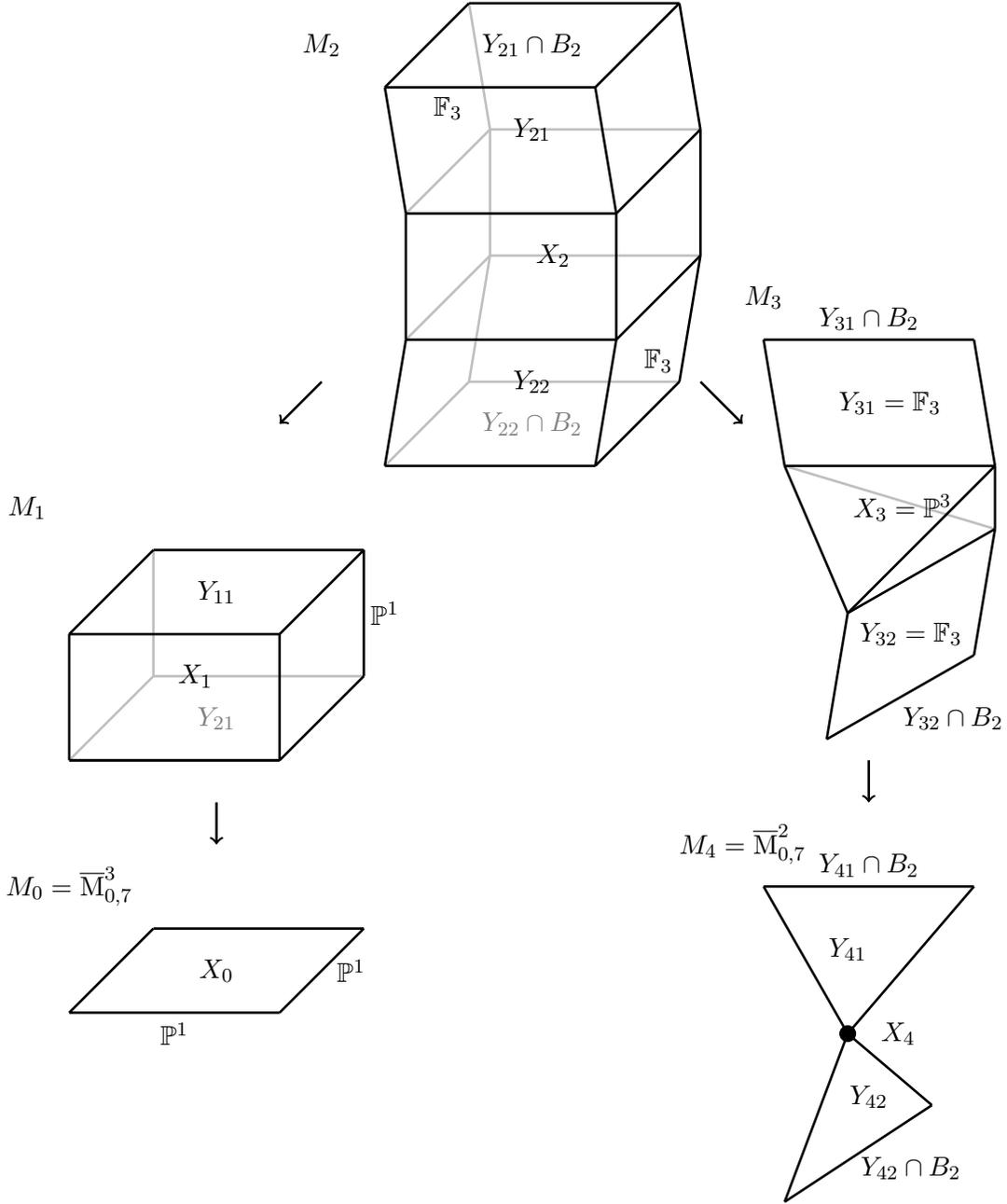
\begin{figure}[!ht]
\begin{tikzpicture}[scale=0.6]
	\draw[line width=1pt] (0,0) -- (5,0);
	\draw[line width=1pt] (5,0) -- (7,2);
	\draw[line width=1pt] (7,2) -- (2,2);
	\draw[line width=1pt] (0,0) -- (2,2);
	\node at (3.5,1) {$X_{0}$};
	\node at (2.5, -0.5) {$\PP^{1}$};
	\node at (6.7,1) {$\PP^{1}$};
	\node at (0, 3) {$M_{0} = \Mzv^{3}$};

	\draw[->][line width=1pt] (3.5, 5) -- (3.5, 4);

	\draw[line width=1pt][lightgray] (0,6) -- (2,8);
	\draw[line width=1pt][lightgray] (2,8) -- (2,11);
	\draw[line width=1pt][lightgray] (2,8) -- (7,8);
	\draw[line width=1pt] (0,6) -- (5,6);
	\draw[line width=1pt] (5,6) -- (7,8);
	\draw[line width=1pt] (7,8) -- (7,11);
	\draw[line width=1pt] (5,6) -- (5,9);
	\draw[line width=1pt] (5,6) -- (0,6);
	\draw[line width=1pt] (2,11) -- (7,11);
	\draw[line width=1pt] (7,11) -- (5,9);
	\draw[line width=1pt] (5,9) -- (0,9);
	\draw[line width=1pt] (0,9) -- (0,6);
	\draw[line width=1pt] (0,9) -- (2,11);
	\node at (3, 8) {$X_{1}$};
	\node at (-1, 12) {$M_{1}$};
	\node at (7.5, 9.5) {$\PP^{1}$};
	\node[gray] at (3.5, 7) {$Y_{21}$};
	\node at (3.5, 10) {$Y_{11}$};

	\draw[->][line width=1pt] (6,15) -- (5,14);
	
	\draw[line width=1pt][lightgray] (15,21) -- (10,21);
	\draw[line width=1pt][lightgray] (10,21) -- (8,19);
	\draw[line width=1pt][lightgray] (7.5, 13) -- (9.5, 15);
	\draw[line width=1pt][lightgray] (9.5,15) -- (14.5, 15);
	\draw[line width=1pt][lightgray] (9.5,15) -- (10, 18);
	\draw[line width=1pt][lightgray] (10,18) -- (15,18);
	\draw[line width=1pt][lightgray] (10,18) -- (8,16);
	\draw[line width=1pt][lightgray] (10,18) -- (10,21);
	\draw[line width=1pt][lightgray] (9.5,24) -- (10,21);
	\draw[line width=1pt] (8, 16) -- (13, 16);
	\draw[line width=1pt] (13, 16) -- (15, 18);
	\draw[line width=1pt] (15,18) -- (15, 21);
	\draw[line width=1pt] (8,19) -- (8,16);
	\draw[line width=1pt] (8,19) -- (13,19);
	\draw[line width=1pt] (13,19) -- (13, 16);
	\draw[line width=1pt] (13,19) -- (15, 21);
	\draw[line width=1pt] (8, 16) -- (7.5, 13);
	\draw[line width=1pt] (7.5, 13) -- (12.5, 13);
	\draw[line width=1pt] (12.5, 13) -- (13, 16);
	\draw[line width=1pt] (12.5, 13) -- (14.5, 15);
	\draw[line width=1pt] (14.5, 15) -- (15, 18);
	\draw[line width=1pt] (8,19) -- (7.5,22);
	\draw[line width=1pt] (13,19) -- (12.5, 22);
	\draw[line width=1pt] (7.5,22) -- (12.5, 22);
	\draw[line width=1pt] (12.5, 22) -- (14.5, 24);
	\draw[line width=1pt] (14.5,24) -- (15,21);
	\draw[line width=1pt] (14.5,24) -- (9.5,24);
	\draw[line width=1pt] (9.5,24) -- (7.5,22);
	\node at (6, 23) {$M_{2}$};
	\node at (11.5, 18) {$X_{2}$};
	\node at (11, 15) {$Y_{22}$};
	\node at (11, 21) {$Y_{21}$};
	\node at (11, 23) {$Y_{21} \cap B_{2}$};
	\node[gray] at (11, 14) {$Y_{22} \cap B_{2}$};
	\node at (9, 21.5) {$\mathbb{F}_{3}$};
	\node at (14, 15.5) {$\mathbb{F}_{3}$};
	
	\draw[->][line width=1pt] (15, 15) -- (16,14);
	
	\draw[line width=1pt][lightgray] (22, 11.5) -- (17,13);
	\draw[line width=1pt] (17, 13) -- (22, 13);
	\draw[line width=1pt] (17, 13) -- (16.5, 16);
	\draw[line width=1pt] (16.5, 16) -- (21.5, 16);
	\draw[line width=1pt] (21.5, 16) -- (22, 13);
	\draw[line width=1pt] (17, 13) -- (18.5, 9.5);
	\draw[line width=1pt] (22,13) -- (18.5, 9.5);
	\draw[line width=1pt] (18.5,9.5) -- (22, 11.5);
	\draw[line width=1pt] (22, 11.5) -- (22,13);
	\draw[line width=1pt] (22, 11.5) -- (21.5, 8.5);
	\draw[line width=1pt] (21.5, 8.5) -- (18, 6.5);
	\draw[line width=1pt] (18, 6.5) -- (18.5, 9.5);
	\node at (19.8, 12) {$X_{3} = \PP^{3}$};
	\node at (19.5, 14.5) {$Y_{31} = \mathbb{F}_{3}$};
	\node at (20, 9) {$Y_{32} = \mathbb{F}_{3}$};
	\node at (16.5, 17) {$M_{3}$};
	\node at (19, 16.5) {$Y_{31} \cap B_{2}$};
	\node at (21, 7) {$Y_{32} \cap B_{2}$};
	
	\draw[->][line width=1pt] (19, 6) -- (19, 5);
	
	\draw[line width=1pt] (16.5, 3) -- (21.5, 3);
	\draw[line width=1pt] (16.5, 3) -- (18.5, -0.5);
	\draw[line width=1pt] (18.5,-0.5) -- (21.5,3);
	\draw[line width=1pt] (18.5,-0.5) -- (17, -4.5);
	\draw[line width=1pt] (18.5,-0.5) -- (20.5, -2.2);
	\draw[line width=1pt] (20.5,-2.2) -- (17,-4.5);
	\fill (18.5, -0.5) circle(0.2);
	\node at (19.7, -0.5) {$X_{4}$};
	\node at (16, 4) {$M_{4} = \Mzv^{2}$};
	\node at (18.5, 1.5) {$Y_{41}$};
	\node at (19, -2) {$Y_{42}$};
	\node at (19, 3.5) {$Y_{41} \cap B_{2}$};
	\node at (20, -3.7) {$Y_{42} \cap B_{2}$};
\end{tikzpicture}
\caption{Decomposition of the second flip $\Mzv^{3} \dashrightarrow \Mzv^{2}$}\label{fig:secondflip}
\end{figure}

On $\Mzv^{3}$, let $X_{0}$ be an irreducible component of $B_{2}^{2}$. Then $X_{0}$ is isomorphic to $\PP^{1} \times \PP^{1}$ and its normal bundle $N_{X_{0}/\Mzv^{3}}$ is isomorphic to $\cO(-2,-1)\oplus \cO(-1,-2)$ (Lemma \ref{lem:descriptionofX}). Note that on $\Mzv^{3}$, since we have blown-up $B_{2}^{3}$, $X_{0}$ is the intersection of exactly two irreducible components of $B_{2}$ and no other irreducible components of $B_{2}$ intersects $X_{0}$. From the computation of the normal bundle, the direct summands $\cO(-2,-1)$ and $\cO(-1,-2)$ correspond to the normal bundle to two irreducible components of $B_{2}$ containing $X_{0}$.

Take the blow-up $M_{1}$ of $M_{0} := \Mzv^{3}$ along $X_{0}$. Then the exceptional divisor $X_{1}$ is isomorphic to $\PP(\cO(-2,-1) \oplus \cO(-1,-2))$. It has two sections $Y_{11}$ and $Y_{12}$, which are intersections with the proper transform of irreducible components of $B_{2}$. The normal bundle $N_{Y_{11}/M_{1}}$ is isomorphic to $\cO(-2,-1) \oplus \cO(1,-1)$ and $N_{Y_{12}/M_{1}} \cong \cO(-1,-2) \oplus \cO(-1,1)$ (Lemma \ref{lem:normalbdlYoi}).

Let $M_{2}$ be the blow-up of $M_{1}$ along $Y_{11} \sqcup Y_{12}$. Let $Y_{21}$ (resp. $Y_{22}$) be the exceptional divisor over $Y_{11}$ (resp. $Y_{12}$). Finally, let $X_{2}$ be the proper transform of $X_{1}$. Since $X_{2}$ is a blow-up of two Cartier divisors $Y_{11}, Y_{12} \subset X_{1}$, $X_{2}$ is isomorphic to $X_{1}$. On the other hand, $Y_{21} \cong \PP(\cO(-2,-1)\oplus \cO(1,-1))$ and $Y_{22} \cong \PP(\cO(-1,-2)\oplus \cO(-1,1))$.

If we fix the first coordinate on $Y_{11}$, then the restriction of $N_{Y_{11}/M_{1}}$ is $\cO(-1) \oplus \cO(-1)$. So its projectivization is $\PP^{1} \times \PP^{1}$. This implies that $Y_{21}$ has another $\PP^{1}$ fibration structure which does not come from $Y_{21} \to Y_{11}$. Moreover, if we restrict $\cO_{Y_{21}}(Y_{21})$ to a fiber, it is isomorphic to $\cO_{\PP^{1}}(-1)$. Therefore we can blow-down this $\PP^{1}$ fibration and the result is smooth. $Y_{22}$ can be contracted in the same way. (But note that the direction of fibrations are different.) Let $M_{3}$ be the blow-down of $Y_{21}$ and $Y_{22}$, and let $Y_{31}$ (resp. $Y_{32}$, $X_{3}$) be the image of $Y_{21}$ (resp. $Y_{22}$, $X_{2}$). Then $Y_{31}$, $Y_{32}$ are isomorphic to $\mathbb{F}_{3}$ and $X_{3}$ is isomorphic to $\PP^{3}$ and $N_{X_{3}/M_{3}} \cong \cO(-3)$ (Lemma \ref{lem:descriptionofXthree}). 

Finally, $X_{3}$ can be contracted to a point $X_{4}$ in the category of algebraic spaces (\cite[Corollary 6.10]{Art70}). Let $M_{4}$ be the contraction. $X_{4}$ is a singular point of $M_{4}$. The image $Y_{41}$ (resp. $Y_{42}$) of $Y_{31} \cong \mathbb{F}_{3}$ (resp. $Y_{32}$) is the contraction of $(-3)$ section, hence it is covered by a single family of rational curves passing through the singular point. Let $\Mzv^{2} := M_{4}$. 

We claim that $\Mzv^{2}$ is the second flip. The argument is standard. There is a small contraction $\phi_{2} : \Mzv^{3} \to \Mzv(\psi - 3K_{\Mzv})$ (Lemma \ref{lem:flippingbase}). For two modifications $\pi_{2} : M_{2} \to \Mzv^{3}$ and $\pi_{2}' : \Mzv^{3} \to \Mzv^{2}$, by rigidity lemma, there is a morphism $\phi_{2}' : \Mzv^{2} \to \Mzv(\psi-3K_{\Mzv})$ such that $\phi_{2} \circ \pi_{2} = \phi_{2}' \circ \pi_{2}'$. We prove that for $D \in (\psi-5K_{\Mzv}, \psi-3K_{\Mzv})$, $D$ is ample on $\Mzv^{2}$ (Lemma \ref{lem:amplenessofD}). Note that it implies the projectivity of $\Mzv^{2}$. In summary, we obtain following result. 

\begin{proposition}\label{prop:secondflip}
The modification $\Mzv^{2}$ is $D$-flip of $\Mzv^{3}$ for $D \in (\psi-5K_{\Mzv}, \psi - K_{\Mzv})$. 
\end{proposition}

Now we show lemmas we mentioned in the outline.

\begin{lemma}\label{lem:descriptionofX}
\begin{enumerate}
	\item On $\Mzv^{3}$, $X_{0} \cong \PP^{1} \times \PP^{1}$.
	\item The normal bundle $N_{X_{0}/\Mzv^{3}}$ is isomorphic to 
	$\cO(-2,-1)\oplus \cO(-1,-2)$. 
\end{enumerate}
\end{lemma}

\begin{proof}
Take an irreducible component of $B_{2}^{2}$ on $\Mzv$, which is isomorphic to $\Mzf$. Let $p$, $q$ be two attaching points. One can also regard $\Mzf$ as a universal family over $\overline{\mathrm{M}}_{0,4} \cong \PP^{1}$ which is also isomorphic to blow-up of $\PP^{1} \times \PP^{1}$ along three diagonal points. Its four sections correspond to 4 marked points for $\Mzf$. Then there are four sections (say $i, j, k$ and $p$) such that three of them are proper transforms of trivial sections and one of them is the proper transform of the diagonal section. We may assume that $p$ is the diagonal section. The normal bundle $N_{\Mzf/\Mzv} \cong \cO(-\psi_{p}) \oplus \cO(-\psi_{q})$. By intersection number computation, one can show that $N_{\Mzf/\Mzv} \cong \pi^{*}(\cO(-2, -1)\oplus \cO(-1, -2)) \otimes \cO(E_{i}+E_{j}+E_{k})$ where $\pi : \Mzf \to \PP^{1} \times \PP^{1}$ is the blow-up along three intersection points of the diagonal section and $E_{i}, E_{j}, E_{k}$ are three exceptional divisors. On $\Mzv$, these three exceptional curves are three components of $B_{2}^{3}$. 

On $\Mzv^{3}$, $X_{0}$ is the blow-up of $\Mzf$ along three divisors and contraction along the different direction. Thus $X_{0}$ is the contraction of three exceptional lines $E_{i}$, $E_{j}$, and $E_{k}$ and it is isomorphic to $\PP^{1} \times \PP^{1}$. This proves (1). 

We denote the proper transform of $X_{0}$ in $\widetilde{\mathrm{M}}_{0,7}^{3}$ by $\widetilde{X}$. Let $\pi_{1} : \widetilde{X} \to \Mzf$, $\pi_{2} : \widetilde{X} \to X_{0}$ be two contractions. (Since $B_{2}^{3} \subset X_{0}$ is a divisor, $\pi_{1}$ is an isomorphism.) Then by the blow-up formula of normal bundles \cite[App. B.6.10.]{Ful98}, $N_{\widetilde{X}/\widetilde{\mathrm{M}}_{0,7}^{3}} \cong \pi_{1}^{*}N_{\Mzf/\Mzv} \otimes \cO(-E_{i}-E_{j}-E_{k}) \cong \pi_{1}^{*}\pi^{*}(\cO(-2,-1) \oplus \cO(-1, -2)) = \pi_{2}^{*}(\cO(-2,-1) \oplus \cO(-1, -2))$. Since the opposite blow-up center is transversal to $X$, $N_{X/\Mzv^{3}} \cong \cO(-2,-1) \oplus \cO(-1,-2)$. 
\end{proof}

\begin{lemma}\label{lem:normalbdlYoi}
The normal bundle $N_{Y_{11}/M_{1}}$ is isomorphic to $\cO(-2,-1) \oplus \cO(1,-1)$. Similarly, $N_{Y_{12}/M_{1}} \cong \cO(-1,-2) \oplus \cO(-1,1)$.
\end{lemma}

\begin{proof}
For a section $Y_{11} = \PP(\cO(-2,-1)) \subset \PP(\cO(-2,-1)\oplus \cO(-1,-2)) = X_{1}$, the normal bundle $N_{X_{1}/M_{1}}|_{Y_{11}} \cong \cO(-2,-1)$ and $N_{Y_{11}/X_{1}} \cong \cO(-1,-2) \otimes \cO(-2,-1)^{*} \cong \cO(1,-1)$. From the normal bundle sequence 
\[
	0 \to N_{Y_{11}/X_{1}} \to N_{Y_{11}/M_{1}} \to N_{X_{1}/M_{1}}|_{Y_{11}} \to 0, 
\]
$N_{Y_{11}/M_{1}}$ is an extension of $N_{X_{1}/M_{1}}|_{Y_{11}}$ by $N_{Y_{11}/X_{1}}$. But $\mathrm{Ext}^{1}(\cO(-2,-1), \cO(1,-1)) \cong H^{1}(\cO(3, 0)) = 0$. Therefore $N_{Y_{11}/M_{1}} \cong \cO(-2,-1) \oplus \cO(1,-1)$. The computation of $N_{Y_{12}/M_{1}}$ is similar.
\end{proof}

\begin{lemma}\label{lem:descriptionofXthree}
\begin{enumerate}
	\item $Y_{31} \cong Y_{32} \cong \mathbb{F}_{3}$.
	\item $X_{3} \cong \PP^{3}$.
	\item $N_{X_{3}/M_{3}} \cong \cO(-3)$.
\end{enumerate}
\end{lemma}

\begin{proof}
Since the restriction of $N_{Y_{21}/M_{2}}$ to $\PP^{1} \times \{*\} \subset Y_{11}$ is isomorphic to $\cO(-2)\oplus \cO(1)$, the restriction of $Y_{21}$ onto the inverse image of $\PP^{1} \times \{*\}$ is $\PP(\cO(-2)\oplus \cO(1)) \cong \mathbb{F}_{3}$. Hence $Y_{31}$ is also isomorphic to Hirzebruch surface $\mathbb{F}_{3}$. This proves (1). 

The divisor $X_{2}$ is isomorphic to $\PP(\cO(-2,-1) \oplus \cO(-1,-2))$. Note that two contracted loci $Y_{21} \cap X_{2}$ (resp. $Y_{22} \cap X_{2}$) has normal bundle $\cO(-1,-2)\otimes \cO(-2,-1)^{*} \cong \cO(1,-1)$ (resp. $\cO(-1,1)$). This is isomorphic to the blow-up of $\PP^{3}$ along two lines $L_{1}$ and $L_{2}$ in general position. Indeed, if we consider the universal (or total) space of all lines intersect $L_{1}$ and $L_{2}$, then naturally it is identified to $\mathrm{Bl}_{L_{1} \cup L_{2}}\PP^{3}$. Thus this blown-up space has a $\PP^{1}$-fibration structure over (both of) exceptional divisor isomorphic to $\PP^{1} \times \PP^{1}$. The normal bundles to two exceptional divisors are $\cO(1,-1)$ and $\cO(-1, 1)$ respectively. Thus $X_{2} \cong \mathrm{Bl}_{L_{1} \cup L_{2}}\PP^{3}$ and we have $X_{3} \cong \PP^{3}$.

For a diagonal embedding $\PP^{1} \hookrightarrow \PP^{1} \times \PP^{1} = X_{0}$, if we restrict to $\PP(\cO(-2,-1) \oplus \cO(-1,-2)) \to X_{0}$, we obtain a trivial bundle $\PP(\cO(-3)\oplus \cO(-3)) \to \PP^{1}$. Take a general constant section $s \hookrightarrow \PP(\cO(-3) \oplus \cO(-3))$. Then the restricted normal bundle $N_{X_{1}/M_{1}}|_{s}$ is isomorphic to $\cO_{\PP^{1}}(-3)$. We may choose $s$ which does not intersect $Y_{ij}$ during modifications. Thus $N_{X_{1}/M_{1}}|_{s} = N_{X_{3}/M_{3}}|_{s}$ and $s$ is a line in $X_{3} \cong \PP^{3}$. Hence $N_{X_{3}/M_{3}} \cong \cO(-3)$. 
\end{proof}

\begin{lemma}\label{lem:flippingbase}
For $D = \psi - 3K_{\Mzv}$, there is a small contraction $\phi_{2}: \Mzv^{3}\to \Mzv(D)$ which contracts a connected component of $B_{2}^{2}$ to a point.
\end{lemma}

\begin{proof}
Since $X_{0}$ is isomorphic to $\PP^{1}\times \PP^{1}$, it is covered by two rational curve classes $\ell_{1} = \PP^{1} \times \{x\}$ and $\ell_{2} = \{y\} \times \PP^{1}$. For a general $x$, $\ell_{1}$ does not intersect the flipping locus of $\Mzv \dashrightarrow \Mzv^{3}$. Moreover, this is a curve class $A$ in Remark \ref{rem:curveA}. So by Corollary \ref{cor:intersection}, $\ell_{1} \cdot D = 0$. By the same reason, $\ell_{2} \cdot D = 0$. Since $\ell_{1}, \ell_{2}$ generates the cone of curves of $\PP^{1} \times \PP^{1}$, $D$ is numerically trivial on $B_{2}^{2}$. Because the only numerically trivial divisor on $B_{2}^{2}$ is a trivial divisor, $D$ does not have any base points on $B_{2}^{2}$. By Proposition \ref{prop:stablebaselocus}, on the outside of $B_{2}^{2}$, there is no base point of $mD$ for $m \gg 0$, too. Thus $D$ is a semi-ample divisor on $\Mzv^{3}$. So there is a regular morphism $\phi_{2}: \Mzv^{3} \to \Mzv^{3}(D) \cong \Mzv(D)$, which contracts $B_{2}^{2}$, a codimension two subvariety to a point. 
\end{proof}

\begin{lemma}\label{lem:amplenessofD}
For $D \in (\psi - 5K_{\Mzv}, \psi - 3K_{\Mzv})$, $D$ is ample on $\Mzv^{2}$. 
\end{lemma}

\begin{proof}
Because it is a contraction of $M_{3}$, which is a projective variety, $\Mzv^{2}$ satisfies the assumption of \cite[Lemma 4.12]{FS11}. Thus we can apply Kleiman's criterion and we will show that for $D \in [\psi-5K_{\Mzv}, \psi-3K_{\Mzv}]$, $D$ is nef. 

Since $mD$ for $m \gg 0$ is base-point-free for all $\Mzv - B_{2}^{2} \cong \Mzv^{2} - Y_{41} \cup Y_{42}$, it is enough to check that for all curve classes on $Y_{41} \cup Y_{42}$, the intersection with $D$ is nonngative. The curve cone of $Y_{41}$ is generated by single rational curve $\ell$, which is the image of a fiber $f$ in $\mathbb{F}_{3}$. So it suffices to compute $D \cdot \ell$. The computation of the intersection number of curve class in $Y_{42}$ is identical.

It is easy to see that $B_{2} \cdot \ell = 1$ from the description of $M_{4}$. To compute $B_{3} \cdot \ell$, we need to keep track the proper transform of $B_{3}$. Note that there are seven irreducible components (say $B_{31}, \cdots, B_{37}$) of $B_{3}$ intersect $X_{0}$. If we write $\mathrm{Pic}(X_{0}) = \langle h_{1}, h_{2}\rangle$ where $h_{1}$ (resp. $h_{2}$) is the curve class of $\PP^{1} \times \{*\}$ (resp. $\{*\} \times \PP^{1}$), three of them ($B_{31}, B_{32}, B_{33}$) are $h_{1}$, other three of them ($B_{34}, B_{35}, B_{36}$) are $h_{2}$, and the other ($B_{37}$) is $h_{1}+h_{2}$ class, which is the diagonal set-theoretically. By keep tracking the proper transforms, one can check that on $M_{3}$, $Y_{31} \subset B_{3i}$ for $i = 1, 2, 3, 7$, $Y_{31} \cap B_{3j} = \PP^{1}= f$ for $j = 4, 5, 6$. Also $X_{3} \cap Y_{3k}$ is a plane for $k = 1, 2, \cdots, 6$, but $X_{3} \cap Y_{37}$ is a quadric containing two skew lines $Y_{31}\cap X_{3}, Y_{32}\cap X_{3}$. 

Analytic locally near $X_{4}$, $M_{4}$ is isomorphic to a cone over degree 3 Veronese embedding of $\PP^{3}$ in $\PP^{19}$, $Y_{41}$ is a cone over a twisted cubic curve, and $M_{3}$ is the blow-up of the conical point. If we take the pull-back of a hyperplane class $H \subset \PP^{20}$ containing $X_{4}$ for $\pi : M_{3} \to M_{4}$, $\pi^{*}H = \widetilde{H} + X_{3}$ where $\widetilde{H}$ is the proper transform of $H$. Note that $\widetilde{H} \cap X_{3} \subset X_{3} \cong \PP^{3}$ is a cubic surface. Therefore $\pi^{*}\pi_{*}B_{3i} = B_{3i} + \frac{1}{3}X_{3}$ for $i =1, \cdots 6$, $\pi^{*}\pi_{*}B_{37} = B_{37}+\frac{2}{3}X_{3}$. Now 
\begin{eqnarray*}
	B_{3}\cdot \ell &=& \pi^{*}B_{3} \cdot f = \sum_{i=1}^{7}B_{3i}\cdot f 
	+ 6 \cdot \frac{1}{3}X_{3}\cdot f + \frac{2}{3}X_{3}\cdot f\\
	&=& (B_{31}+B_{32}+B_{33}+B_{37})\cdot f + \frac{8}{3}.
\end{eqnarray*}
For a 1-dimensional fiber $f'$ of $Y_{21} \to Y_{11}$, $f'$ maps to $f$ by $Y_{21} \to Y_{31}$. By projection formula for $\rho : M_{2} \to M_{3}$, 
\[
	B_{3i}\cdot f = \rho^{*}B_{3i} \cdot f' = \widetilde{B}_{3i}\cdot f' 
	+ Y_{21}\cdot f' = Y_{21} \cdot f' = -1
\]
if we denote the proper transform of $B_{3i}$ by $\widetilde{B}_{3i}$. Therefore
\[
	B_{3}\cdot \ell = -4 + \frac{8}{3} = -\frac{4}{3}.
\]

For $D = \psi - aK_{\Mzv}$, $D \equiv \frac{5+a}{3}B_{2}+2B_{3}$ by Corollary \ref{cor:divisorclasscomputation}. So $D \cdot \ell = \frac{a-3}{3}$ and it is nonnegative if $a \ge 3$. 
\end{proof}

\subsection{Divisorial contraction}\label{ssec:divcontraction}

The last birational model $\Mzv^{1}$ is a divisorial contraction.

\begin{lemma}
Let $D = \psi - 5K_{\Mzv}$. Then $D$ is a semi-ample divisor on $\Mzv^{2}$. 
\end{lemma}

\begin{proof}
By Proposition \ref{prop:stablebaselocus}, the stable base locus is contained in the union of the proper transform of $B_{2}$ and $\cup Y_{4i}$. By the proof of Lemma \ref{lem:amplenessofD}, $D$ is ample on $\cup Y_{4i}$. So it suffices to show that $D$ is semi-ample on the proper transform of $B_{2}$.

Since $D$ is in the closure of the ample cone of $\Mzv^{2}$, $D$ is nef. In particular, if $B_{I}$ is an irreducible (equivalently on $\Mzv^{2}$, connected) component of $B_{2}$, $D|_{B_{I}}$ is nef. But on $\Mzv$, $B_{I} \cong \Mzs$ so it is a Mori dream space. Since the proper transform of $B_{I}$ on $\Mzv^{2}$ is a flip of $B_{I}$, it is a Mori dream space, too. Thus for $m \gg 0$, $mD|_{B_{I}}$ is base-point-free. Thus $\mathbf{B}(D) = \emptyset$ on $\Mzv^{2}$ and it is semi-ample.
\end{proof}

Let $\Mzv^{1} = \Mzv(\psi - 5K_{\Mzv}) = \Mzv^{2}(\psi-5K_{\Mzv})$. Since $B_{2}$ is covered by a curve class $C_{5}$ such that $C_{5}\cdot D = 0$, so $\Mzv^{1}$ is a divisorial contraction of $\Mzv^{2}$. 

\begin{proposition}\label{prop:divcontraction}
For $D \in (B_{2}, \psi - 5K_{\Mzv}]$, $\Mzv(D) \cong \Mzv^{1}$.
\end{proposition}
\begin{proof}
Note that for $D \in (B_{2}, \psi-5K_{\Mzv}]$, $D \equiv (\psi - 5K_{\Mzv}) + cB_{2}$ for some $c \ge 0$. Because $B_{2}$ is an exceptional divisor for $\phi_{1} : \Mzv^{2} \to \Mzv^{1}$, $\Mzv(D) \cong \Mzv^{2}(D) \cong \Mzv^{2}(\psi-5K_{\Mzv}) \cong \Mzv^{1}$.
\end{proof}

\section{KKO compactification}\label{sec:KKOcompactification}

In this section, we give a review of KKO compactification of moduli of curves of genus $g$ in a smooth projective variety $X$, which will be used to describe a modular interpretation of $\Mzv^{3}$ in next section. For the detail of its construction, consult the original paper of Kim, Kresch, and Oh (\cite{KKO14}). 

\subsection{FM degeneration space}

Fix a nonsingular projective variety $X$. Let $X[n]$ be the Fulton-MacPherson space of $n$ distinct ordered points in $X$. It is a compactification of the moduli space of $n$ ordered distinct points on $X$, which is obviously $X^{n}\setminus \Delta$. See \cite{FM94} for the construction and its geometric properties. $X[n]$ has a universal family $\pi : X[n]^{+} \to X[n]$ and $n$ disjoint universal sections $\sigma_{i} : X[n] \to X[n]^{+}$ for $1 \le i \le n$. 

For a point $p \in X[n]$, the fiber $\pi^{-1}(p)$ is a possibly reducible variety, whose irreducible components are smooth and equidimensional. As an abstract variety, $\pi^{-1}(p)$ can be constructed in the following manner. Set $X_{0} := X$. Take a point $x_{0} \in X$ and blow-up $X_{0}$ along $x_{0}$. Let $\widetilde{X}_{0} := \mathrm{Bl}_{x_{0}}X_{0}$ and $E_{1}$ be the exceptional divisor, which is naturally isomorphic to $\PP(T_{x_{0}}X_{0})$. Now consider the compactified tangent space $\PP T := \PP(T_{x_{0}}X_{0}\oplus \CC)$, which has a subvariety $\PP(T_{x_{0}}X_{0}) \cong \PP T - T_{x_{0}}X_{0}$. Glue $\widetilde{X}_{0}$ and $\PP T$ along $\PP(T_{x_{0}}X_{0})$ and let $X_{1}$ be the result. 

We are able to continue this construction, by taking a nonsingular point $x_{1} \in X_{1}$ and construct $X_{2}$ in a same way. If we repeat this procedure several times, we inductively obtain $X_{k}$, which is a reducible variety. $\pi^{-1}(p)$ is isomorphic to $X_{k}$ for some $k \ge 0$ and some $x_{0}$, $x_{1}$, $\cdots$, $x_{k-1}$. Note that there is a natural projection $X_{k} \to X$. In can be extended to a canonical morphism $\pi_{X} : X[n]^{+} \to X$.

\begin{remark}
\begin{enumerate}
\item The singular locus of $X_{k}$ is isomorphic to a union of disjoint $\PP^{r-1}$'s. 
\item Naturally the dual graph of $X_{k}$ is a tree with a root. The proper transform of $X_{0}$ corresponds to the root. A non-root component is called a \textbf{screen}. The \textbf{level} of an irreducible component of $X_{k}$ is defined by the number of edges from the root to the vertex representing the component. 
\item If an irreducible component $Y$ of $X_{k}$ does not contains any $x_{i}$, then $Y \cong \PP^{r}$. $Y$ is called an \textbf{end component}.
\item If an irreducible component $Z$ of $X_{k}$ is not the root component and it contains only two singular loci, then $Z \cong \mathrm{Bl}_{p}\PP^{r}$, which is a ruled variety. $Z$ is called a \textbf{ruled component}.
\end{enumerate}
\end{remark}

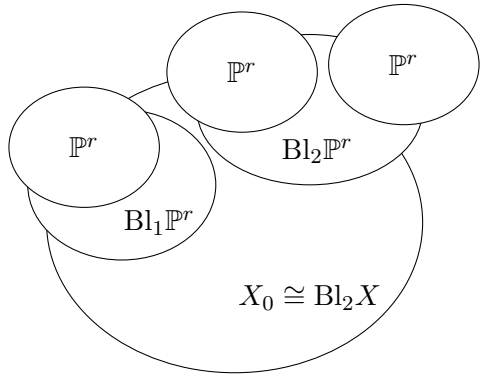
\begin{figure}[!ht]
\begin{tikzpicture}[scale=0.5]
	\draw (0,0) ellipse (5 and 4);
	\fill[color=white] (2,3) ellipse (3 and 2);
	\draw (2,3) ellipse (3 and 2);
	\fill[color=white] (-3,1) ellipse (2.5 and 2);
	\draw (-3,1) ellipse (2.5 and 2);
	\fill[color=white] (0.2,4) ellipse (2 and 1.6);
	\draw (0.2,4) ellipse (2 and 1.6);
	\fill[color=white] (4.5,4.2) ellipse (2 and 1.6);
	\draw (4.5,4.2) ellipse (2 and 1.6);
	\fill[color=white] (-4,2) ellipse (2 and 1.6);
	\draw (-4,2) ellipse (2 and 1.6);
	\node (root) at (2,-2){$X_{0} \cong \mathrm{Bl}_{2}X$};
	\node (1) at (-2, 0){$\mathrm{Bl}_{1}\PP^{r}$};
	\node (2) at (-4,2){$\PP^{r}$};
	\node (3) at (2.2,1.9){$\mathrm{Bl}_{2}\PP^{r}$};
	\node (4) at (4.5,4.2){$\PP^{r}$};
	\node (5) at (0.2,4){$\PP^{r}$};
\end{tikzpicture}
\caption{An example of FM degeneration space}\label{fig:FMspace}
\end{figure}

\begin{definition}\cite[Definition 2.1.1]{KKO14}
A pair $(\pi_{W/B} \to B, \pi_{W/X} : W \to X)$ is called a Fulton-MacPherson degeneration space of $X$ over a scheme $B$ (or an FM degeneration space of $X$ over $B$) if:
\begin{itemize}
\item $W$ is an algebraic space;
\item \'Etale locally it is a pull-back of the universal family $\pi: X[n]^{+} \to X[n]$. That is, there is an \'etale surjective morphism $B' \to B$ from a scheme $B'$, $n > 0$ and a Cartesian diagram
\[
	\xymatrix{W|_{B'} \ar[r] \ar[d] & X[n]^{+} \ar[d]\\
	B' \ar[r] & X[n]}
\]
where the pull-back of $\pi_{W/X}$ to $W|_{B'}$ is equal to $W|_{B'} \to X[n]
^{+} \to X$.
\end{itemize}
\end{definition}

Let $W$ be an FM space over $\CC$. An automorphism of $W/X$ is an automorphism $\varphi : W \to W$ fixing the root component, or equivalently, $\pi_{W/X} \circ \varphi = \pi_{W/X}$. If $W \ncong X$, $\mathrm{Aut}(W/X)$ is always positive dimensional. More precisely, for an end component $Y$ of $W$, the automorphism fixing all $W$ except $Y$ is isomorphic to $\CC^{r} \rtimes \CC^{*}$, the group of homotheties. Also for a ruled component $Z$ of $W$, the automorphism fixing $W$ except $Z$ is isomorphic to $\CC^{*}$. The other irreducible components do not contribute to a non-trivial automorphism of $W/X$. 

We leave a useful lemma to show several geometric properties of KKO compactifications. 

\begin{lemma}\label{lem:unifamilycontraction}
For $m > n$, there is a commutative diagram
\[
	\xymatrix{X[m]^{+} \ar[r] \ar[d] & X[n]^{+}\ar[d]\\
	X[m] \ar[r] & X[n].}
\]
Two vertical maps are universal families, and the horizontal maps obtained by forgetting $m-n$ marked points and stabilizing. 
\end{lemma}

\begin{proof}
By induction, it suffices to show for $m = n +1$ case. Note that $X[n+1]$ is obtained by taking a blow-up of $X[n]^{+}$ along the image of $n$ sections (\cite[195p]{FM94}). On the other hand, $X[n]^{+}$ is constructed by taking iterated blow-ups of $X[n] \times X$. Hence we have a commutative diagram
\[
	\xymatrix{X[n+1]^{+} \ar[r] \ar[d] & X[n]^{+} \ar[d]\\
	X[n+1] \times X \ar[d] & X[n] \times X\ar[d]\\
	X[n+1] \ar[uur] \ar[r] & X[n].}
\]
\end{proof}

\subsection{Stable unramified maps}

\begin{definition}\cite[Definition 3.1.1]{KKO14}\label{def:unramifiedstablemap}
A collection of data
\[
	((C, x_{1}, x_{2}, \cdots, x_{n}), \pi_{W/X} : W \to X, f : C \to W)
\]
is called an \textbf{$n$-pointed stable unramified map} of type $(g, \beta)$ to an FM degeneration space $W$ of $X$ if:
\begin{enumerate}
\item $(C, x_{1}, x_{2}, \cdots, x_{n})$ is an $n$-pointed prestable curve with arithmetic genus $g$;
\item $\pi_{W/X}: W \to X$ is an FM degeneration space of $X$ over $\CC$;
\item $(\pi_{W/X} \circ f)_{*}[C] = \beta \in A_{1}(X)$;
\item $f^{-1}(W^{sm}) = C^{sm}$, where $Y^{sm}$ is the smooth locus of $Y$.
\item $f|_{C^{sm}}$ is unramified everywhere;
\item $f(x_{i})$ for $1 \le i \le n$ are distinct;
\item At each nodal point $p \in C$, there are coordinates 
\[
	\hat{\cO}_{p} \cong \CC[[x,y]]/(x,y) \mbox{ and } 
	\hat{\cO}_{f(p)} \cong \CC[[z_{1}, \cdots ,z_{r+1}]]/(z_{1}z_{2})
\]
such that $\hat{f}^{*}: \CC[[z_{1}, \cdots ,z_{r+1}]]/(z_{1}z_{2}) \to \CC[[x,y]]/(xy)$ maps $z_{1}$ to $x^{m}$ and $z_{2}$ to $y^{m}$ for some $m \in \NN$. 
\item There are finitely many automorphisms $\sigma : C \to C$ such that $\sigma(x_{i}) = x_{i}$ for $1 \le i \le n$ and $f \circ \sigma = \varphi \circ f$ for some $\varphi \in \mathrm{Aut}(W/X)$. 
\end{enumerate}
\end{definition}

We can define the \textbf{level} of an irreducible component $C'$ of $C$ by the level of the component of $W$ containing $f(C')$. A component $C'$ with a positive level is called a \textbf{ghost component}. 

\begin{remark}
The last condition about the finiteness of automorphisms can be described conditions on end components and ruled components in the following way. A map $f : C \to W$ has a finite automorphism group if and only if:
\begin{itemize}
\item For each end component $Y$ of $W$, the number of marked points on $Y$ is at least two or there is an irreducible component $D$ of $C$ such that $f(D) \subset Y$ and $\deg f(D) \ge 2$;
\item For each ruled component $Z$ of $W$, there is at least one marked point on $Z$ or there is an irreducible component $D \subset C$ such that $f(D)$ is not contained in a ruling. 
\end{itemize}
\end{remark}

\begin{definition}\cite[Definition 3.2.1]{KKO14}\label{def:familyofunramifiedmaps}
A collection of data
\[
	((\pi : \cC \to B, \sigma_{1}, \cdots, \sigma_{n}), (\pi_{W/B} : W \to B, 
	\pi_{W/X}: W \to X), f : \cC \to W)
\]
is called a \textbf{$B$-family of $n$-pointed stable unramified maps} of type $(g, \beta)$ to FM degeneration spaces of $X$, if:
\begin{enumerate}
\item $(\pi : \cC \to B, \sigma_{1}, \sigma_{2}, \cdots, \sigma_{n})$ is a family of $n$-pointed genus $g$ prestable curves over $B$;
\item $(\pi_{W/B}: W \to B, \pi_{W/X} : W \to X)$ is an FM degeneration space of $X$ over $B$;
\item Over each geometric point of $B$, the data restricted to the fiber is a stable unramified map of type $(g, \beta)$ to an FM degeneration space of $X$;
\item For every geometric point $b \in B$, if $p \in C_{b}$ is a nodal point, then there are two identifications 1) $\hat{\cO}_{f(p)} \cong \hat{\cO}_{\pi_{W/B}(p)}[[z_{1}, z_{2}, \cdots, z_{r+1}]]/(z_{1}z_{2} - t)$ for some $t \in \hat{\cO}_{\pi_{W/B}(p)}$ and 2) $\hat{\cO}_{p} \cong \hat{\cO}_{\pi(p)}[[x, y]]/(xy - t')$ for some $t' \in \hat{\cO}_{\pi(p)}$ such that $\hat{f}^{*}(z_{1}) = \alpha_{1}x^{m}$, $\hat{f}^{*}(z_{2}) = \alpha_{2}y^{m}$ for some $m \in \NN$, $\alpha_{1}, \alpha_{2} \in \hat{\cO}_{p}^{*}$, and $\alpha_{1}\alpha_{2} \in \hat{\cO}_{\pi(p)}$. 
\end{enumerate}
\end{definition}

Let $\overline{\cU}_{g,n}(X, \beta)$ be the fibered category of $n$-pointed unramified stable maps to FM degeneration spaces of $X$ of type $(g, \beta)$. 

\begin{theorem}\cite[Corollary 3.3.3]{KKO14}
The fibered category $\overline{\cU}_{g,n}(X, \beta)$ is a proper Deligne-Mumford stack of finite type. 
\end{theorem}

As in the title of this section, we will call $\overline{\cU}_{g,n}(X, \beta)$ as the \textbf{KKO compactification} of moduli space of embedded curves. By Keel-Mori theorem, we have a coarse moduli space $\overline{U}_{g,n}(X, \beta)$ in the category of algebraic spaces. 

\subsection{Some geometric properties}

In this section, we explain several geometric/functorial properties of $\overline{\cU}_{g,n}(X, \beta)$. 

As in the case of moduli space of ordinary stable maps, there are several functorial maps. Let $\overline{\cM}_{g,n}(X, \beta)$ be the moduli stack of stable maps (\cite{KM94}). 

\begin{proposition}\label{prop:stabilization}
There is a functorial morphism 
\[
	S : \overline{\mathcal{U}}_{g,n}(X, \beta) \to 
	\overline{\mathcal{M}}_{g,n}(X, \beta).
\]
\end{proposition}
\begin{proof}
Let
\[
	((\pi : \cC \to B, \sigma_{1}, \cdots, \sigma_{n}), 
	(\pi_{W/B} : W \to B, \pi_{W/X}: W \to X), f : \cC \to W)
\]
be a $B$-family of $n$-pointed stable unramified maps of type $(g, \beta)$ to FM degeneration spaces of $X$. Then we have $((\pi : \cC \to B, \sigma_{1}, \cdots,  \sigma_{n}), \pi_{W/X} \circ f : \cC \to X)$, which is a flat family of maps from $n$-pointed curves to $X$. By running relative MMP with respect to $\omega_{\cC/B} + \sum \sigma_{i}$, we can stabilize $\pi_{W/X} \circ f$ and obtain 
\[
	((\bar{\pi} : \overline{\cC} \to B, \bar{\sigma}_{1}, \cdots, 
	\bar{\sigma}_{n}), \bar{f} : \overline{\cC} \to X).
\]
These two steps are both functorial, we can obtain the desired morphism $S$. 
\end{proof}

\begin{proposition}\label{prop:evaluation}
There are functorial morphisms
\[
	ev_{i} : \overline{\mathcal{U}}_{g,n}(X, \beta) \to X
\]
for $1 \le i \le n$.
\end{proposition}
\begin{proof}
Indeed $ev_{i} = e_{i} \circ S : \overline{\mathcal{U}}_{g,n}(X, \beta) \to \overline{\mathcal{M}}_{g,n}(X, \beta) \to X$ where $e_{i}$ be the $i$-th evaluation map for the ordinary moduli space of stable maps. 
\end{proof}

\begin{proposition}\label{prop:forgetful}
For any $T \subset [n]$, there is a functorial morphism
\[
	F : \overline{\cU}_{g,n}(X, \beta) \to
	\overline{\cU}_{g,T}(X, \beta)
\]
obtained by forgetting all marked points with indices in $[n]-T$ and stabilizing. 
\end{proposition}
\begin{proof}
It suffices to show the existence of $F : \overline{\cU}_{g,n}(X, \beta) \to \overline{\cU}_{g,n-1}(X, \beta)$ which forgets the last marked point. For a family 
\[
	((\pi : \cC \to B, \sigma_{1}, \cdots, \sigma_{n}), (\pi_{W/B} : W \to B, 
	\pi_{W/X}: W \to X), f : \cC \to W)
\]
of $n$-pointed stable unramified maps over $B$, if we forget the last section $\sigma_{n}$, then the remaining collection of data
\begin{equation}\label{eqn:familyunramifiedstablemaps}
	((\pi : \cC \to B, \sigma_{1}, \cdots, \sigma_{n-1}), 
	(\pi_{W/B} : W \to B, 
	\pi_{W/X}: W \to X), f : \cC \to W)
\end{equation}
is also a family of $(n-1)$-pointed unramified stable maps unless 
\begin{enumerate}
\item For a fiber of $b \in B$, there is an end component $Y$ of $W_{b}$ such that for every components $D_{i}$ of $\cC_{b}$ maps to $Y$, $D_{i}$ is a rational curve maps to a line injectively, and there are exactly two marked points $\sigma_{n}(b)$ and $\sigma_{k}(b)$ lie on $\cup D_{i}$ or;
\item For a fiber of $b \in B$, there is a ruled component $Z$ of $W_{b}$ such that for every components $D_{j}$ of $\cC_{b}$ maps to $Z$, the image of $D_{j}$ is a ruling and only $\sigma_{n}(b)$ lies on $\cup D_{j}$. Note that $D_{j}$ is a rational curve, because it is a ramified cover of $\PP^{1}$ which has exactly two branch points. 
\end{enumerate}
Note that only one of these two cases may happen on a fiber. 

We can stabilize the family \eqref{eqn:familyunramifiedstablemaps} in the following way. Suppose that \'etale locally, the target space $\pi_{W/B} : W \to B$ comes from the Cartesian diagram
\[
	\xymatrix{W|_{B'} \ar[r] \ar[d] & X[m]^{+}\ar[d]\\
	B' \ar[r] & X[m]}
\]
for some $m > 0$ and an \'etale map $B' \to B$. We will modify the family locally, so for simplicity, we may assume that there is a unique connected closed subset $U \subset T$ such that for $b \in U$, the fiber has an end component $Y$ of $W_{b}$ with property (1). Also, we may assume that there is a unique connected closed subset $V \subset T$ such that for $b \in V$, there is a rule component $Z$ of $W_{b}$ with property (2). Over $U$ (resp. $V$), the non-stable end components (resp. ruled components) form a family of irreducible components of $W|_{U}$ (resp. $W|_{V}$). 

Let $\tau_{1}, \tau_{2}, \cdots, \tau_{m} : B' \to W|_{B}$ be the pull-back of universal sections $\sigma_{1}, \sigma_{2}, \cdots, \sigma_{m} : X[m] \to X[m]^{+}$. Let $I \subset [m]$ be the index set of sections such that $i \in I$ if and only if $\tau_{i}$ is on the non-stable end component. Pick any $j \in I$ and let $J := I - \{j\}$. Now we have a forgetting map $X[m] \to X[m - |J|]$ forgetting all section in $J$. There is also a contraction map $X[m]^{+} \to X[m-|J|]^{+}$ on the universal family by Lemma \ref{lem:unifamilycontraction}. Take the pull-back of the universal family $X[m-|J|]^{+} \to X[m-|J|]$ by $B' \to X[m] \to X[m-|J|]$. Then we have a family $W'|_{B'} \to B'$ of FM degeneration spaces and there is a morphism $W|_{B'} \to W'|_{B'}$.
\[
	\xymatrix{\cC|_{B'} \ar[r]^{f} \ar[rrd] \ar[rdd] & 
	W|_{B'} \ar[dd]\ar[rr] \ar[rd] && 
	X[m]^{+}\ar[rd]\ar[dd]\\
	&& W'|_{B'} \ar[dd] \ar[rr] && X[m-|J|]^{+} \ar[dd]\\
	& B' \ar@{=}[rd] \ar[rr] && X[m]\ar[rd]\\
	&& B' \ar[rr] && X[m-|J|]}
\]
Now there are several irreducible components of $\cC_{b}$ for $b \in V$, which are all tails, such that $f : \cC|_{B'} \to W|_{B'} \to W'|_{B'}$ is not finite. By using the standard stabilizing of the domain curve (running the relative MMP over $W|_{B'}$ for $(\cC|_{B'}, \omega_{C/B'}+\sum \sigma_{i})$), we can contract these irreducible components. 

After performing this procedure finite times, we can remove all non-stable end components and getting new family of maps $\overline{\cC}|_{B'} \to W'|_{B'}$. Note that this procedure does not depend on the choice of $m$, $B' \to X[m]$ and $J \subset [m]$. We may replace $\cC|_{B'}$ by $\overline{\cC}|_{B'}$ and $W|_{B'}$ by $W'|_{B'}$ for a notational convenience. 

The contraction of a non-stable ruled component in (2) is similar. Take $K \subset [m]$ such that $i \in K$ if and only if $\tau_{i}$ is on the non-stable ruled component. Take the forgetting map $X[m] \to X[m- |K|]$. Bu taking the pull-back of the universal family $X[m-|K|]^{+} \to X[m-|K|]$, we have a family $W''|_{B'} \to B'$, and a $B'$-morphism $W|_{B'} \to W''|_{B'}$. By contracting all non-finite components using standard relative MMP technique, we obtain a family of finite maps $\overline{\cC}|_{B'} \to W''|_{B'}$ over $B'$. 

We claim that the result is a family of unramified stable maps. Except (7) on Definition \ref{def:unramifiedstablemap}, all other conditions are simple observations of contracting procedures. If we contract a non-stable end component $Y$ of the target, because we contract all irreducible components on the domain whose image lie on $Y$, there is no relevant singular points on the domain anymore. Furthermore, if we contract a non-stable ruled component $Z$ of the target, then an irreducible component $C_{i}$ of the domain maps to $Z$ has only two ramification points at two singular points of the domain on $C_{i}$. Moreover, since $C_{i} \cong \PP^{1}$, the ramification indices at two singular points are equal. Thus after the contraction of the component, the stabilized map has the property (7). 
\end{proof}

\begin{proposition}\label{prop:mortotangent}
Let $X$ be a smooth projective variety. Then there is a morphism 
\[
	T: \overline{\cU}_{g,n}(X, \beta) \to 
	\bigsqcup_{\beta' \in A_{1}(\PP(TX), \ZZ)}\overline{\cM}_{g,n}(\PP(TX), \beta')
\]
where $\PP(TX)$ be the projectivized tangent bundle of $X$.
\end{proposition}

\begin{proof}
This is a direct consequence of \cite[Lemma 3.2.4]{KKO14}. For a family 
\[
	((\pi : \cC \to B, \sigma_{1}, \cdots, \sigma_{n}), (\pi_{W/B} : W \to B,
	\pi_{W/X} : W \to X), f : \cC \to W),
\]
we have a family of maps $\tilde{f}: \cC \to \PP(TX)$, which is a unique extension of the projectivized tangent map $\PP(Tf) : \cC^{sm} \to \PP(TX)$. By stabilizing the domain as usual, we obtain a family of stable maps $\bar{f} : \overline{\cC} \to \PP(TX)$. 
\end{proof}

\begin{remark}
For a ghost component $C'$ of the domain $C$, the map $\PP(Tf) : C' \to \PP(TX)$ can be described in the following way. Each screen (after blowing down all higher level screens) is identified with $\PP(T_{x}X\oplus \CC)$ for some $x \in X$. For a smooth point $p \in C'$, $\PP(Tf)(p) = T_{p}C' \cap \PP(T_{x}X)$, where $\PP(T_{x}X) \subset \PP(T_{x}X \oplus \CC)$ is the `hyperplane at infinity'. Therefore it is a projection of the tangent variety of $C'$. If $C'$ is a rational normal curve of degree $d$ in $\PP^{r}$ with $r \ge d$, then $\deg \PP(Tf)(C') = 2d - 2$ (\cite[245p.]{Har95}). 
\end{remark}

\begin{example}
If $X = \PP^{d}$, then the Chow ring of $\PP(T\PP^{d})$ is 
\[
	A^{*}(\PP(T\PP^{r}), \ZZ) \cong \ZZ[H, \zeta]/
	\left\langle H^{d+1}, 
	\sum_{i=0}^{d}{d+1\choose i}H^{i}\zeta^{d-i}\right\rangle
\]
where $H$ is the pull-back of hyperplane class $h$ in $\PP^{d}$ and $\zeta = c_{1}(\cO_{\PP(T\PP^{d})}(1))$.

We claim that for the connected component of $\cUznpdd$ containing smooth rational normal curves in $\PP^{d}$, $\beta'$ in Proposition \ref{prop:mortotangent} is $dH^{d-1}\zeta^{d-1}+(d+2)(d-1)H^{d}\zeta^{d-2}$ if $d \ge 2$. First of all, $\deg H^{d}\zeta^{d-1} = 1$. From the combination of two relations, we can deduce $H^{d-1}\zeta^{d}+(d+1)H^{d}\zeta^{d-1} = 0$ so $\deg H^{d-1}\zeta^{d} = -(d+1)$. Since $H^{d-1}\zeta^{d-1}$ and $H^{d}\zeta^{d-2}$ form a basis of $A_{1}(\PP(T\PP^{d}), \ZZ)$, $\beta'$ is a linear combination of them. For a stable unramified map $f : C \to \PP^{d}$ where $f(C)$ is a smooth rational curve of degree $d$ in $\PP^{d}$, $T(f)(C) = \PP(TC) \subset \PP(T\PP^{d})$, thus the restriction of the tautological subbundle to $T(f)(C)$ is $TC \cong \cO_{\PP^{1}}(2)$. Hence $T(f)(C) \cdot \zeta = -2$. On the other hand, from the projection formula $T(f)(C) \cdot H = f(C) \cdot h = d$. Therefore from a simple calculation, we obtain $\beta' = dH^{d-1}\zeta^{d-1} + (d+2)(d-1)H^{d}\zeta^{d-2}$.

From now, in this paper we denote $aH^{d-1}\zeta^{d-1} + bH^{d}\zeta^{d-2}$ by \textbf{$\mathbf{(a,b)}$-class}. 
\end{example}

\subsection{Deformation theory}\label{ssec:deformation}

The dimensions of the deformation and obstruction spaces of $\overline{\cU}_{g,n}(X, \beta)$ can by computed indirectly by using Olsson's deformation theory of log schemes (\cite{Ols05}). For a family 
\[
	((\pi : \cC \to B, \sigma_{1}, \cdots, \sigma_{n}), 
	(\pi_{W/B} : W \to B, \pi_{W/X} : W \to X), f : \cC \to W)
\]
of $n$-pointed stable unramified maps over $B$, we can introduce natural log structures $M^{\cC/B}$ on $\cC$, $M^{W/B}$ on $W$, and $N^{\cC/B}$ and $N^{W/B}$ on $B$ such that $(\cC, M^{\cC/B}) \to (B, N^{\cC/B})$ and $(W, M^{W/B}) \to (B, N^{W/B})$ are log smooth morphisms. We obtain a canonical log structure $N$ on $B$ by taking monoid push-out $N^{\cC/B}\oplus_{N'}N^{W/B}$ where $N'$ is the submonoid of $N^{\cC/B}\oplus N^{W/B}$ generated by $(m\cdot \log t', \log t)$ for each nodal point of $\cC$ (for the definition of $m, t, t'$, see Definition \ref{def:familyofunramifiedmaps}.). 

We have a stack $\cB$ of $n$-pointed prestable curves, FM degeneration spaces with $n$ distinct smooth points, fine log schemes, and pairs of morphisms of log structures
\[
	((\cC \to B, (\sigma_{1}, \cdots, \sigma_{n})), 
	(W \to B, (\tau_{1},  \cdots, \tau_{n})),
	(B, N), N^{\cC/B} \to N, N^{W/B} \to N).
\]
The relative tangent/obstruction spaces for $\overline{\cU}_{g,n}(X, \beta) \to \cB$ are described by cohomology groups. Suppose that $B = \spec R$ for a Noetherian $\CC$-algebra $R$ and $\widetilde{R}$ is a square-zero extension of $R$ by $I$. Let $\widetilde{B} = \spec \widetilde{R}$. Also suppose that $\widetilde{\cC}$ (resp. $\widetilde{W}$) is an extension of $\cC$ (resp. $W$) over $\widetilde{B}$. Let $\widetilde{N}$ be the extension of $N$ over $\widetilde{B}$ with two extensions $N^{\widetilde{\cC}/\widetilde{B}}\to \widetilde{N}$ and $N^{\widetilde{W}/\widetilde{B}} \to \widetilde{N}$. Then the obstruction for a compatible extension of a stable unramified map is an element of $H^{1}(\cC, f^{*}T^{\dagger}_{W}(-\sum \sigma_{i})\otimes I)$ and if the obstruction vanishes, the compatible extensions identified with $H^{0}(\cC, f^{*}T^{\dagger}_{W}(-\sum \sigma_{i})\otimes I)$ (\cite[Proposition 5.1.1]{KKO14}). Here $T_{W}^{\dagger}$ means the log tangent sheaf. 

On the other hand, there is a log version of moduli space of stable log maps $\overline{\cU}_{g,n}^{log}(X, \beta)$, constructed in \cite{Kim10}. There is a commutative diagram 
\[
	\xymatrix{\overline{\cU}_{g,n}^{log}(X, \beta) \ar[d]_{\phi}\ar[rd]
	\\ \overline{\cU}_{g,n}(X, \beta) \ar[r] & \cB}
\]
where $\phi$ is a virtual normalization map (\cite{LM12}). $\phi$ is finite and degree one. 

Let $B^{\dagger}$ be the log scheme $(B, N)$. Let $\cC^{\dagger}$ be the minimal log curve induced by $N^{\cC/B} \to N$ (\cite[3.5]{Kim10}) and let $W^{\dagger}$ be the semi-stable log scheme induced by $N^{W/B} \to N$ (\cite[4.3]{Kim10}). Let $\mathrm{Aut}_{I}(\cC^{\dagger}\times_{B^{\dagger}}W^{\dagger})$ be the set of automorphisms of the trivial extensions of $\cC^{\dagger}\times_{B^{\dagger}}W^{\dagger}$ over $\spec (\widetilde{R}, \widetilde{N})$, whose restriction to $B^{\dagger}$ is the identity. And let $\mathrm{Def}_{I}(\cC^{\dagger}\times_{B^{\dagger}}W^{\dagger})$ be the set of isomorphism classes of $I$-extensions of log schemes over $B^{\dagger}$.There is an $R$-module exact sequence 
\[
	0 \to \mathrm{Aut}_{I}(\cC^{\dagger}\times_{B^{\dagger}}W^{\dagger})
	\to \mathrm{RelDef}(f) = H^{0}(\cC, f^{*}T_{W^{\dagger}/B^{\dagger}}(-\sum \sigma_{i})\otimes_{\cO_{B}}I) \to \mathrm{Def}(f)
\]
\[
	\to \mathrm{Def}_{I}(\cC^{\dagger}\times_{B^{\dagger}}W^{\dagger})
	\to \mathrm{RelOb}(f) = H^{1}(\cC, f^{*}T_{W^{\dagger}/B^{\dagger}}(-\sum \sigma_{i})\otimes_{\cO_{B}}I) \to \mathrm{Obs}(f) \to 0
\]
(\cite[Section 7.1]{Kim10}). 

Now consider $B = \spec \CC$ case. If $H^{1}(C, f^{*}T_{W}^{\dagger}(-\sum \sigma_{i})) = 0$, then $\phi$ is a local isomorphism, thus $\mathrm{RelOb}(f) = 0$ as well. Also $\mathrm{Obs}(f) = 0$ hence both $\overline{\cU}_{g,n}^{log}(X, \beta)$ and $\overline{\cU}_{g,n}(X, \beta)$ are smooth. Thus we have:

\begin{lemma}\label{lem:smoothstack}
Let $((C, x_{1}, x_{2}, \cdots, x_{n}), \pi_{W/X} : W \to X, f : C \to W)$ be a stable unramified map over $\spec \CC$. If $H^{1}(C, f^{*}T^{\dagger}_{W}(-\sum \sigma_{i})) = 0$, then $\overline{\cU}_{g,n}(X, \beta)$ is smooth at the point.
\end{lemma}

\section{$\Mzv^{3}$ as a parameter space}\label{sec:modular}

In this section, we discuss a moduli theoretic interpretation of $\Mzv^{3}$, the first flip of $\Mzv$. 

In a recent result \cite{Smy13}, Smyth described a systematic classification of modular compactifications $\overline{\cM}_{g,n}(\mathcal{Z})$ of $\cM_{g,n}$, which can be described in term of certain combinatorial data $\mathcal{Z}$. They are moduli spaces of pointed curves with (possibly) worse singularities. In the case of $g = 0$, he obtained a complete classification of such compactifications (\cite[Theorem 1.21]{Smy13}). When $g = 0$, all such compactifications are obtained by contracting some irreducible components of parameterized curves and obtaining new arithmetic genus 0 singularities there. Because a singularity of arithmetic genus 0 does not have a positive dimensional moduli, all such compactifications are (usually small) contractions of $\Mzn$. Therefore if we want to describe a moduli theoretic meaning of a flip of $\Mzn$, then it must not be a moduli of pointed curves with a certain singularity type. In other words, it is not a substack of the stack of all pointed curves (\cite[Appendix B]{Smy13}). 

From the description of $\Mzv^{3}$, we have several clues on the possible moduli theoretic meaning of it.
\begin{enumerate}
\item The reduction map $\phi : \Mzv \to V_{A}^{3}$ contracts F-curves of type $F_{1,2,2,2}$. The image of a contracted F-curve corresponds to a pointed rational curve $(C, x_{1}, x_{2}, \cdots, x_{7})$ which has three irreducible components and they meet at a triple nodal singularity. $\phi$ forgets the cross-ratio of four special points on the spine of $F_{1,2,2,2}$. 
\item A connected component of the exceptional fiber of the contraction $\phi_{3}': \Mzv^{3} \to V_{A}^{3}$ is isomorphic to $\PP^{2}$. 
\end{enumerate}

Note that the image of $F_{1,2,2,2}$ is exactly the locus of non-nodal (non-Gorenstein as well) curves on $V_{A}^{3}$ (See Example \ref{ex:Veronesequotient}.). From (2), we may guess that $\Mzv^{3}$ is a moduli space of pointed curves parameterized by $V_{A}^{3}$, with some additional structure on non-Gorenstein singularities. 

\begin{question}\label{que:infinitesimalstructure}
What kind of infinitesimal structure can we give on non-Gorenstein singularities?
\end{question}

Note that $V_{A}^{3}$ is defined as a GIT quotient of an incidence variety in the product $\Mzzptt \times (\PP^{3})^{7}$. At least as parameter spaces in a weak sense, we are able to construct many new birational models of $\Mzv$ by using incidence varieties. For example, if we introduce additional factors such as $\mathbb{G}\mathrm{r}(1,3)^{7}$ which has the information about a tangent direction at each point, and take the GIT quotient (with an appropriate linearization) of the incidence variety in 
\[
	\Mzzptt \times (\PP^{3})^{7} \times \mathbb{G}\mathrm{r}(1,3)^{7},
\]
then we may have a resolution of $V_{A}^{3}$. Also we may replace a factor by another modular variety. For instance it would be interesting if we consider the Fulton-MacPherson space $\PP^{3}[7]$ instead of $(\PP^{3})^{7}$. But in our situation, we need to find a parameter space which does fit into the picture of Mori's program for $\Mzv$. Thus a refined question is the following:

\begin{question}\label{que:fitintodiagram}
Which of them does fit into the diagram $\phi_{3}': \Mzv^{3} \to V_{A}^{3}$?
\end{question}

To answer this question, we will use KKO compactification we have discussed in Section \ref{sec:KKOcompactification}. 

Let $\cUznpdd$ be the KKO compactification of the space of $n$-pointed rational normal curves in $\PP^{d}$ and let $\Uznpdd$ be its coarse moduli space. Similarly, let $\cMznpdd$ be the moduli stack of ordinary stable maps and $\Mznpdd$ be its coarse moduli space. We have the following commutative diagram:
\[
	\xymatrix{\Uzvptt \ar[r]^(0.4){F'} \ar[d]_{S} & 
	\Uzzptt \times (\PP^{3})^{7} \ar[d]^{S'}\\
	\Mzvptt \ar[r]^(0.4){F} & \Mzzptt \times (\PP^{3})^{7}}.
\]
The vertical map $S$ is the stabilization map $S$ in Proposition \ref{prop:stabilization}, and $S' = S \times \mathrm{id}$. $F$ is the product of a forgetful map and evaluation maps for the moduli space of stable maps, and $F' = F \times \prod ev_{i}$ is that of KKO compactifications (Proposition \ref{prop:forgetful} and Proposition \ref{prop:evaluation}).

Let $I \subset \Mzzptt \times (\PP^{3})^{7}$ be the incidence variety parameterizes $(f : C \to \PP^{3}, x_{1}, \cdots, x_{7})$ such that $x_{i} \in \im f$ for all $i$. It is straightforward to check that $I = \im \;\phi$. From the description of $V_{A}^{3}$ in Section \ref{ssec:Veronesequotients}, $V_{A}^{3} \cong I\git_{L} \SL_{4}$ with a suitable linearization $L$ which is a restriction of a linearized ample line bundle on $\Mzzptt \times (\PP^{3})^{7}$. Note that with respect to $L$, the stability coincides with the semi-stability. Let $I^{s}$ be the stable locus. 

Suppose that we have an incidence variety $J \subset \Uzzptt \times (\PP^{3})^{7}$. We would like to show that $J\git \SL_{4} \cong \Mzv^{3}$ for an appropriate choice of a linearization. The choice of the linearization is standard. For any $G$-equivariant projective morphism between two quasi-projective varieties $f : X \to Y$ and a linearization $L$ on $Y$ such that $Y^{ss}(L) = Y^{s}(L)$, there is a linearization $L'$ on $X$ such that 
\[
	X^{ss}(L') = X^{s}(L') = f^{-1}(Y^{s}(L))
\]
(\cite[Section 3]{Kir85}, \cite[Theorem 3.11]{Hu96}). With respect to this linearization, there is a quotient map $\overline{S} : J\git_{L'}\SL_{4} \to I\git_{L}\SL_{4} \cong V_{A}^{3}$. Thus if we carefully analyze the fiber of $\overline{S}$, then we may prove that $J\git_{L'}\SL_{4} \cong \Mzv^{3}$. 

But there are a few technical difficulties on this approach. Because the geometry of $\cUznprd$ is very complicate, there are few results on its geometric properties. For instance, $\cUznprd$ is not irreducible in general, the connectedness is unknown, and we don't know about the projectivity of its coarse moduli space $\Uznprd$ even for $n = 0$ and $r = d = 3$. Furthermore, we don't have a nice modular description nor the deformation theory for the `main component' of $\cUznprd$. So we are unable to apply the above standard approach. Thus here we will use an ad-hoc approach.

Let $\cMzzptt^{nd} \subset \cMzzptt$ be the substack of stable maps non-degenerated image and let $\Mzzptt^{nd} \subset \Mzzptt$ be its coarse moduli space. Since $(f : C \to \PP^{3}) \in \cMzzptt^{nd}$ has no nontrivial automorphism, $\cMzzptt^{nd} = \Mzzptt^{nd}$ is a smooth open subvariety of $\Mzzptt$. Let $\cUzzptt^{nd} := S^{-1}(\cMzzptt^{nd})$ for the stabilization map in Proposition \ref{prop:stabilization} and let $\Uzzptt^{nd}$ be its coarse moduli space. 

\begin{lemma}\label{lem:ndlocusissmooth}
The open subset $\Uzzptt^{nd} \subset \Uzzptt$ is a smooth algebraic space. 
\end{lemma}

\begin{proof}
First of all, we will show that $\cUzzptt^{nd}$ is a smooth stack. Because every object $(f : C \to W) \in \cUzzptt^{nd}$ is injective, it has no nontrivial automorphism. Thus $\cUzzptt^{nd} = \Uzzptt^{nd}$ and the latter one is also smooth as an algebraic space. 

Since $\cMzzptt$ is a smooth Deligne-Mumford stack, it suffices to check that the smoothness at a map $(f : C \to W) \in \cUzzptt^{nd}$ lying on the locus that $S : \cUzzptt^{nd} \to \cMzzptt$ is not an isomorphism. If the target space $W$ is $\PP^{3}$ then there is no ghost component and hence $(f : C \to W = \PP^{3})$ is already an object in $\cMzzptt^{nd}$. Since $\pi \circ f(C)$ is degenerated in $\PP^{3}$, for any screen (after blowing-down all higher level screens) $Y \cong \PP(T_{x}\PP^{3}\oplus \CC)$, $f(C) \cap \PP(T_{x}\PP^{3})$ is a union of reduced points. If there is an end component $Y \cong \PP(T_{x}\PP^{3}\oplus \CC) \subset W$ of level one such that $\PP(T_{x}\PP^{3}) \cap f(C)$ is a set of two reduced points, then every ghost conic on $Y$ are equivalent to each other and hence there is no non-trivial moduli of them. Hence $\cUzzptt^{nd}$ is not locally isomorphic to $\cMzzptt^{nd}$ along the locus parametrizes a map $(f : C \to W)$ where the domain has three tails $C_{1}$, $C_{2}$, $C_{3}$ and there is a ghost spine $C_{4}$. There are three possibilities. See Figure \ref{fig:ghostspine}.

\begin{enumerate}
\item The spine $C_{4}$ is a level one smooth cubic ghost component. 
\item $C_{4} = C_{4,1} \cup C_{4,2} \cup C_{4,3}$ is a chain of rational curves. $C_{4,1}$ has level one and degree two, $C_{4,3}$ has level one and degree one. Finally $C_{4,2}$ has level two and degree two. 
\item $C_{4} = C_{4,1} \cup \cdots \cup C_{4,5}$ is a chain of rational curves. $C_{4,1}, C_{4,3}, C_{4,5}$ are level one linear ghost components and $C_{4,2}, C_{4,4}$ are level two degree two ghost components on two different end components.
\end{enumerate}

\begin{figure}[!ht]
\begin{tikzpicture}[scale=0.6]
	\draw (-0.3,0.3) ellipse (4 and 4);
	\draw[line width=1pt] (-0.5,0.5) -- (-0.5,3.7);
	\draw[line width=1pt] (-0.5,0.5) -- (-3.6,-1.8);
	\draw[line width=1pt] (-0.5,0.5) -- (3,-2);
	\fill[color=white] (-0.5,0.5) ellipse (2 and 1.8);
	\draw (-0.5,0.5) ellipse (2 and 1.8);
	\begin{scope}
		\clip (-0.5,0.5) ellipse (2 and 1.8);
		\draw[line width=1pt] (-2.05, -0.65) .. controls(-2, 1) and (1, 1) .. (1.1, -0.6);
		\draw[line width=1pt] (-0.5,2.5) -- (-1,-0.7);
	\end{scope}
	\fill[color=white] (-1.3,0.7) ellipse (0.8 and 0.7);
	\draw (-1.3,0.7) ellipse (0.8 and 0.7);
	\draw[line width=1pt] (-0.9,0.1) .. controls (-1.7,1) and (-1.1,1) .. (-0.5, 0.6);
	\node (C41) at (0.3,-0.3){$C_{4,1}$};
	\node (C43) at (0,1.2){$C_{4,3}$};
	\node (C42) at (-1.9,0.7){$C_{4,2}$};
	\node (X) at (2,3){$\PP^{3}$};

	\draw (9.7,0.3) ellipse (4 and 4);
	\draw[line width=1pt] (9.5,0.5) -- (9.5,3.7);
	\draw[line width=1pt] (9.5,0.5) -- (6.4,-1.8);
	\draw[line width=1pt] (9.5,0.5) -- (13,-2);
	\fill[color=white] (9.5,0.5) ellipse (2.5 and 2.2);
	\draw (9.5,0.5) ellipse (2.5 and 2.2);
	\draw[line width=1pt] (9.5, 2.8) -- (9.2, -0.1);
	\draw[line width=1pt] (11.45, -0.85) -- (9, 1.5);
	\fill[color=white] (9,1.5) ellipse (0.8 and 0.7);
	\draw (9,1.5) ellipse (0.8 and 0.7);
	\draw[line width=1pt] (7.55, -0.9) -- (9, -0.5);
	\fill[color=white] (9,-0.5) ellipse (0.8 and 0.7);
	\draw (9,-0.5) ellipse (0.8 and 0.7);
	\draw[line width=1pt] (9.4, 2.1) .. controls (8.8,1.8) and (8.3,1) .. (9.6, 1);
	\draw[line width=1pt] (9.2, 0.2) .. controls (9.5, -0.3) and (9.2,-0.8) .. (8.2, -0.7);
	\node (C433) at (10.2,2.3){$C_{4,3}$};
	\node (C411) at (7.5,-0.4){$C_{4,1}$};
	\node (C44) at (8,1.5){$C_{4,4}$};
	\node (C422) at (9.5, -1){$C_{4,2}$};
	\node (C45) at (11, 0){$C_{4,5}$};
	\node (Y) at (12, 3){$\PP^{3}$};
\end{tikzpicture}
\caption{Ghost spines of type (2) and (3)}\label{fig:ghostspine}
\end{figure}
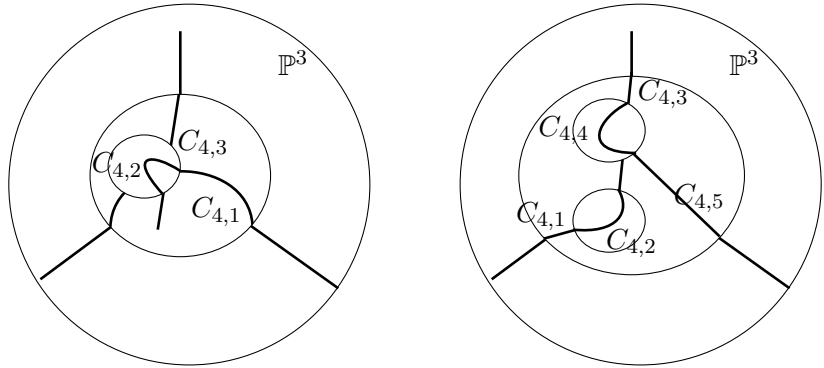

In each case, we are able to show the smoothness by computing the vanishing of the relative obstruction space (See Section \ref{ssec:deformation}). Recall that the relative obstruction is lying on 
\[
	H^{1}(C, f^{*}T_{W}^{\dagger})
\]
where $T_{W}^{\dagger}$ is the logarithmic tangent space of $W$ (\cite[Proposition 5.1.1]{KKO14}). If we decompose $C$ into the union of irreducible components $\cup C_{j}$ and if we denote $f|_{C_{j}}$ by $f_{j}$, then from the short exact sequence 
\[
	0 \to f^{*}T_{W}^{\dagger} \to \bigoplus_{j} f_{j}^{*}T_{W}^{\dagger} \to 
	\bigoplus_{\{j\ne k\}}f^{*}T_{W}^{\dagger}|_{C_{j}\cap C_{k}} \to 0
\]
and the derived long exact sequence
\[
	\bigoplus_{j} H^{0}(C_{j}, f_{j}^{*}T_{W}^{\dagger}) \to 
	\bigoplus_{\{j\ne k\}}f^{*}T_{W}^{\dagger}|_{C_{j}\cap C_{k}} \to 
	H^{1}(C, f^{*}T_{W}^{\dagger}) \to 
	\bigoplus_{j} H^{1}(C_{j}, f_{j}^{*}T_{W}^{\dagger}),
\]
it suffices to show 1) $H^{1}(C_{j}, f_{j}^{*}T_{W}^{\dagger}) = 0$ and 2) the surjectivity of $\bigoplus_{j} H^{0}(C_{j}, f_{j}^{*}T_{W}^{\dagger}) \to 
\bigoplus_{\{j\ne k\}}f^{*}T_{W}^{\dagger}|_{C_{j}\cap C_{k}}$. 

Each irreducible component $C_{j}$ is lying on an irreducible component $V$ of $W$. If $V$ is an end component (which is isomorphic to $\PP^{3}$), then we have an Euler sequence
\[
	0 \to \cO_{V} \to \cO_{V}(1)^{3} \oplus \cO_{V} \to T_{W}^{\dagger}|_{V} \to 0,
\]
and its pull-back 
\begin{equation}\label{eqn:endcomponent}
	0 \to \cO_{C_{j}} \to \cO_{C_{j}}(d)^{3}\oplus \cO_{C_{j}} \to f_{j}^{*}T_{W}^{\dagger}
	\to 0,
\end{equation}
where $d = \deg C_{j}$. Since $H^{1}(\PP^{1}, \cO_{V}(k)) = 0$ for all $k \ge -1$, we have $H^{1}(C_{j}, f_{j}^{*}T_{W}^{\dagger}) = 0$. If $V$ is a root component, then we have
\begin{equation}\label{eqn:rootcomponent}
	0 \to \cO_{V}(-E) \to \pi^{*}\cO_{\PP^{3}}(1)(-E)^{4} \to T_{W}^{\dagger}|_{V} \to 0,
\end{equation}
where $E$ is the exceptional divisor on the root component. Note that for all $f$ above, $E$ is irreducible. Since $f(C_{j})$ is a line intersects $E$, $H^{1}(C_{j}, f_{j}^{*}(\pi^{*}\cO_{\PP^{3}}(1)(-E))) = H^{1}(C_{j},\cO) = 0$. Finally, if $V$ is a screen which is not an end component, we have
\begin{equation}\label{eqn:middlecomponent}
	0 \to \cO_{V}(-E) \stackrel{\iota}{\to} 
	\pi^{*}\cO_{\PP^{3}}(1)(-E)^{3} 
	\oplus \cO_{V}(-E) \to T_{W}^{\dagger}|_{V} \to 0
\end{equation}
where $E$ is the union of exceptional divisors on $V$. In above cases, the component $f(C_{j})$ on $V$ is a conic intersecting an exceptional divisor or a line intersecting one or two exceptional divisors. In any cases, $H^{1}(C_{j}, f_{j}^{*}(\pi^{*}\cO_{\PP^{3}}(1)(-E))) = 0$ thus $H^{1}(C_{j}, f_{j}^{*}(\pi^{*}\cO_{\PP^{3}}(1)(-E)^{3}\oplus \pi^{*}\cO_{V}(-E))) \cong H^{1}(C_{j}, f_{j}^{*}(\cO_{V}(-E)))$. Thus $H^{1}(\iota)$ is surjective and $H^{1}(C_{j}, f_{j}^{*}(T_{W}^{\dagger}|_{V})) = 0$. 

For the surjectivity of 
\[
	\bigoplus_{j} H^{0}(C_{j}, f_{j}^{*}T_{W}^{\dagger}) \to 
	\bigoplus_{\{j\ne k\}}f^{*}T_{W}^{\dagger}|_{C_{j}\cap C_{k}},
\]
we will show a slightly stronger statement: for any level $\ell$ component $C_{j}$ with $\ell = 0, 2$, 
\[
	H^{0}(C_{j}, f_{j}^{*}T_{W}^{\dagger}) \to \bigoplus_{\{\ell(C_{k}) 
	= 1\}}T_{W}^{\dagger}|_{C_{j} \cap C_{k}}
\]
is surjective. If we denote the intersection point $C_{j} \cap C_{k}$ with $\ell(C_{k}) = 1$ by $x_{k}$, then it suffices to show $H^{1}(C_{j}, f_{j}^{*}(T_{W}^{\dagger}(-\sum x_{k}))) = 0$. For a level zero component, which has a unique $x_{k}$, from \eqref{eqn:rootcomponent} we have 
\[
	0 \to \cO_{C_{j}}(-2) \to \cO_{C_{j}}(-1)^{4} \to 
	f_{j}^{*}T_{W}^{\dagger}|_{C_{j}}(-x_{k}) \to 0.
\]
So $H^{1}(C_{j}, f_{j}^{*}T_{W}^{\dagger}|_{C_{j}}(-x_{k})) = 0$. For a level two component, which has two $x_{k}$'s, from \eqref{eqn:endcomponent} we have 
\[
	0 \to \cO_{C_{j}}(-2) \to \cO_{C_{j}}^{3}\oplus \cO_{C_{j}}(-2)
	\to f_{j}^{*}T_{W}^{\dagger}(-\sum x_{k}) \to 0.
\]
We get the vanishing of $H^{1}(C_{j}, f_{j}^{*}T_{W}^{\dagger}(-\sum x_{k}))$ in a similar manner. 
\end{proof}

Let $J^{s} := {S'}^{-1}(I^{s})$ and $J$ be the closure of $J^{s}$ in $\Uzzptt \times (\PP^{3})^{7}$. Then $J$ is the main component of the `incidence subspace' in $\Uzzptt \times (\PP^{3})^{7}$. $J$ and $J^{s}$ are both $\SL_{4}$-invariant subspaces. 

\begin{lemma}\label{lem:schemeness}
\begin{enumerate}
\item The algebraic space $J^{s}$ is a quasi-projective scheme.
\item There is a linearization $L'$ on $J^{s}$ such that for every closed point $x \in J^{s}$, there is a section $s \in H^{0}(J^{s}, L^{m})$ such that $s(x) \ne 0$. In other words, $(J^{s})^{ss}(L') = J^{s}$. 
\end{enumerate}
\end{lemma}

\begin{proof}
By local computation, we can check that the tangent map in Proposition \ref{prop:mortotangent}
\[
	T :  \Uzzptt^{nd} \to \overline{\mathrm{M}}_{0,0}(\PP(T\PP^{3}), (3, 10))
\]
is quasi-finite. Indeed, it may not be injective when $f : C \to W$ has a ghost component of degree 3. Take a rational normal curve $N$ in a non-rigid $\PP^{3} = \{[x:y:z:w]\}$ passing through three coordinate points on the infinite plane $\{x = 0\}$. By using an automorphism of $\PP^{3}$, we may assume that $N$ passes through $p = [1:0:0:0]$. Furthermore, if we fix the image of the tangent map at $p$, or equivalently, the tangent direction at $p$, we have a 2-dimensional family of rational normal curves. We can take an explicit 2-dimensional versal family, for instance, 
\[
	f_{a,b}(s:t) = [(t-3s)(t-s)(t-2s)s: t(at-s)(t-2s)s:
	t(t-s)(4t-s)(t-2s):t(bt-2)(2t-s)s].
\]
By using a computer algebra system, it is straightforward to check that $\PP(Tf_{a,b})([1:0]) = [1:-1:1]$ is independent from $a$ and $b$, but for two $(a,b) \ne (a',b')$, the tangent vectors to $\PP(Tf_{a,b})(\PP^{1})$ and $\PP(Tf_{a',b'})(\PP^{1})$ at $[1:-1:1]$ are different. Thus $T$ is analytic locally injective if $f$ has an irreducible ghost component. The remaining cases are easy to check. 

Since the target of $T$ is a scheme, $\Uzzptt^{nd}$ is a scheme by \cite[Corollary II.6.16]{Knu71}. Furthermore, $\Uzzptt$ is proper and $\overline{\mathrm{M}}_{0,0}(\PP(T\PP^{3}), (3, 10))$ is separated. Thus $T$ is a proper morphism (\cite[Corollary II.4.8]{Har77}). Hence $T$ (restricted to $\Uzzptt^{nd}$) is finite (\cite[Theorem 8.11.1]{EGA4-3}). Thus $T$ is projective (\cite[Corollary 6.1.11]{EGA2}) hence $\Uzzptt^{nd}$ is quasi-projective.

Note that $J^{s} \subset \Uzzptt^{nd} \times (\PP^{3})^{7}$. Since $J^{s}$ is a locally closed subspace of a quasi-projective scheme, it is quasi-projective, too. This proves (1). 

Note that we have a commutative diagram
\[
	\xymatrix{J^{s} \ar[r] \ar[d] & 
	\overline{\mathrm{M}}_{0,0}(\PP(T\PP^{3}), (3,10)) \times 
	(\PP^{3})^{7}\ar[d]^{F}\\
	I^{s} \ar[r] & \Mzzptt \times (\PP^{3})^{7}.}
\]
Since $F$ is a projective morphism, by \cite[Theorem 3.11]{Hu96}, there is a linearization $L'$ on $X := \overline{\mathrm{M}}_{0,0}(\PP(T\PP^{3}), (3,10)) \times (\PP^{3})^{7}$ such that $X^{ss}(L') = X^{s}(L') = F^{-1}((\Mzzptt \times (\PP^{3})^{7})^{s}(L))$. Since $I^{s}$ is in the stable locus of $\Mzzptt \times (\PP^{3})^{7}$, $J^{s}$ maps to the stable locus of $X$. Therefore the pull-back of $L'$ to $J^{s}$ is the linearization we want to find. 
\end{proof}

Therefore by gluing the categorial quotients of affine $\SL_{4}$-invariant subschemes, we obtain a well-defined quotient scheme $J^{s}/\SL_{4}$.

\begin{definition}
The \textbf{formal GIT quotient} $J\git \SL_{4}$ is $J^{s}/\SL_{4}$.
\end{definition}

\begin{remark}
Note that if $\Uzzptt$ is a projective scheme, then for a standard choice of linearization $L'$ on $\Uzzptt \times (\PP^{3})^{7}$, $J\git_{L'}\SL_{4} \cong J^{s}/\SL_{4}$. So far, we don't know the projectivity of $\Uznprd$. We will investigate geometric properties of this moduli space in forthcoming papers. 
\end{remark}

\begin{lemma}\label{lem:normality}
The locus $J^{s}$ is normal. 
\end{lemma}

\begin{proof}
Set $J(0) = \Uzzptt^{nd}$ and for $n \in \NN$, let $J(n) = \{ ((f : C \to W), x_{1}, x_{2}, \cdots, x_{n})\;|\; x_{i} \in \pi \circ f(C)\} \subset \Uzzptt^{nd} \times (\PP^{3})^{n}$ for $\pi : W \to \PP^{3}$. We claim that $J(n)$ is normal. Note that $J(0)$ is normal by Lemma \ref{lem:ndlocusissmooth}. 

Let $p_{n} : J(n) \to J(n-1)$ be the projection map forgetting the last point. Then for any point $((f : C \to W), x_{1}, x_{2}, \cdots, x_{n-1}) \in J(n-1)$, the fiber is isomorphic to $\pi\circ f(C) \subset \PP^{3}$. Since the Hilbert polynomial $P_{\pi \circ f(C)}(m) = 3m+1$ is constant, $p_{n}$ is flat by \cite[Theorem III.9.9]{Har77}. 

Note that a general fiber of $p_{n}$ is smooth because a general element of $J(n-1)$ parametrizes a smooth rational curve. So $J(n)$ is regular in codimension one if $J(n-1)$ is. Also since all fibers are curves, it automatically satisfies Serre's condition $S_{2}$. Therefore $J(n)$ satisfies $S_{2}$ by \cite[Corollary 6.4.2]{EGA4-2}. By Serre's criterion, $J(n)$ is normal if $J(n-1)$ is. 

Since $J^{s}$ is an open subset of $J(7)$, we have the desired result.
\end{proof}

Now we prove the second main result of this paper. 

\begin{theorem}\label{thm:modulardescription}
The formal GIT quotient $J\git \SL_{4}$ is isomorphic to $\Mzv^{3}$.
\end{theorem}

\begin{proof}
Let $\Mzvptt^{s}=F^{-1}(I^{s}) \subset \Mzvptt$ and let $\Uzvptt^{s}=S^{-1}(\Mzvptt^{s}) \subset \Uzvptt$. We have the following diagram:
\[
	\xymatrix{\Uzvptt^{s} \ar[d]_{S} \ar[rd]_{g} \ar@{-->}[rrd] & \\ 
	\Mzvptt^{s} \ar[d]_{F} \ar[r]_(0.6){/\SL_{4}}
	& \Mzv \ar[d]_{\phi} 
	& \widetilde{\mathrm{M}}_{0,7}^{3}
	\ar[l]^{\pi_{3}} \ar[d]^{\pi_{3}'}\\
	I^{s} \ar[r]^{/\SL_{4}} &V_{A}^{3} & \Mzv^{3}\ar[l]_{\phi_{3}'}}
\]
We first show that there is a morphism $\tilde{g} : \Uzvptt^{s} \to \widetilde{\mathrm{M}}_{0,7}^{3}$. Because $\pi_{3}$ is the blow-up along F-curves of type $F_{1,2,2,2}$, from the universal property of blow-up, it is enough to show that $g^{-1}(F_{1,2,2,2})$ is a Cartier divisor in $\Uzvptt^{s}$.

Let $Z^{0} \subset \Uzzptt^{nd}$ be the locally closed subvariety parametrizes $f : C \to W$ such that the domain $C$ has three tails $C_{1}$, $C_{2}$, $C_{3}$ of degree one and an irreducible spine $C_{0}$ which is a ghost component of level one. Let $Z$ be the closure of $Z^{0}$. To obtain $f \in Z^{0}$, we need to choose three lines $C_{1}$, $C_{2}$, and $C_{3}$ on $\PP^{3}$ meet at a point, and a cubic rational normal curve $C_{0}$ in a \emph{non-rigid} $\PP^{3}$ which passes through three points at rigid $\PP^{2} \subset \PP^{3}$. Thus the dimension of $Z^{0}$ is $3+3\cdot 2 + (12-3\cdot 2)-4 = 11$. Hence $Z^{0}$ and $Z$ have codimension one in $\Uzzptt^{nd}$. Because $\Uzzptt^{nd}$ is smooth (Lemma \ref{lem:ndlocusissmooth}), $Z$ is a Cartier divisor. On the other hand, for $F : \Uzvptt \to \Uzzptt$, $F(\Uzvptt^{s}) \subset \Uzzptt^{nd}$ since $\pi \circ f(C)$ is non-degenerated for all $f: C \to W$ in $\Uzvptt^{s}$. Finally, for the forgetful map $F : \Uzvptt^{s} \to \Uzzptt^{nd}$, it is straightforward to check that $g^{-1}(F_{1,2,2,2}) = F^{-1}(Z)$. Therefore $g^{-1}(F_{1,2,2,2})$ is a Cartier divisor as well. Thus we have a morphism $\tilde{g} : \Uzvptt^{s} \to \widetilde{\mathrm{M}}_{0,7}^{3}$. Let $\bar{g}= \pi_{3}' \circ \tilde{g} : \Uzvptt^{s} \to \Mzv^{3}$. 

The forgetful map $F' : \Uzvptt^{s} \to \Uzzptt \times (\PP^{3})^{7}$ factors through $J^{s}$, because $S'\circ F'(\Uzvptt^{s}) = F \circ S(\Uzvptt^{s}) = I^{s}$ and $J^{s} = {S'}^{-1}(I^{s})$. We have an algebraic fiber space $\Uzvptt^{s} \to J^{s}$ because $J^{s}$ is normal (\cite[Proof of Corollary III.11.4]{Har77}). The only possible exceptional curve $E$ for $\Uzvptt^{s} \to J^{s}$ is obtained by varying a unique marked point on a ghost component, hence varying the cross-ratio of them. $E$ is contracted by $\bar{g} : \Uzvptt^{s} \to \Mzv^{3}$ because $\bar{g} = \pi_{3}' \circ \tilde{g}$ and $\pi_{3}' : \widetilde{\mathrm{M}}_{0,7}^{3} \to \Mzv^{3}$ forgets the cross-ratio. Therefore there is a morphism $Q: J^{s} \to \Mzv^{3}$ (\cite[Proposition II.5.3]{Kol96}). Finally, because it is $\SL_{4}$-equivariant, there is a quotient map $\overline{Q}: J\git\SL_{4} = J^{s}/\SL_{4} \to \Mzv^{3}$ and a commutative diagram
\[
	\xymatrix{J\git\SL_{4}\ar[r]^{\overline{Q}}\ar[d] 
	& \Mzv^{3}\ar[d]^{\phi_{3}'}\\
	I \git_{L}\SL_{4} \ar[r]^{\cong} & V_{A}^{3}.}
\]
On a point $x$ of the exceptional locus of $\phi_{3}':\Mzv^{3} \to V_{A}^{3}$, from a dimension counting, it is straightforward to check that the inverse image $\overline{Q}^{-1}(x)$ does not have a positive dimensional moduli. Also on the outside of the exceptional locus, they are isomorphic. Thus $\overline{Q}$ is a quasi-finite birational morphism to a smooth variety. So it is an isomorphism by \cite[Proposition III.9.1]{Mum99}. 
\end{proof}

\begin{remark}\label{rem:modulardescription}
We may describe an object in $J\git\SL_{4}$ in an intrinsic way. For $(f : C \to W) \in \cUzzptt^{nd}$, suppose that the image of $\pi \circ f : C \to W \to \PP^{3}$ has a non-Gorenstein singularity at $x \in \im \pi \circ f(C)$. There are three irreducible components meet at $x$. The level one component $Y = \PP(T_{x}\PP^{3}\oplus \CC)$ of $W$ at $x$ can be regarded as a compactified non-rigid tangent space $\PP(T_{x}C\oplus \CC)$, because the three irreducible components generate $\PP^{3}$. Hence the infinitesimal structure we can give on the non-Gorenstein singularity $x \in C$, as an answer for Questions \ref{que:infinitesimalstructure} and \ref{que:fitintodiagram}, is a ghost rational cubic curve (and its degeneration) on a compactified non-rigid tangent space of $C$ at $x$. 
\end{remark}

\begin{remark}\label{rem:largen}
\begin{enumerate}
\item It would be very interesting if one can define $J\git\SL_{4}$ as a moduli stack directly, instead of describing it as a quotient stack of a certain moduli stack.
\item The similar modular flip appears for every $n \ge 7$. For example, if we consider a $D$-filp for the total boundary divisor $B$ on $\Mzn$, then the flipping locus contains the  locus covered by $F_{1, i,j,k}$ where $i,j,k \ge 2$. Therefore it is inevitable to study such flips in general, if we would like to study full symmetric Mori's program for $\Mzn$.
\end{enumerate}
\end{remark}

%%%%%%%%%%%%%%%%%%%%%%%%%%%%%%%%%%%%%
%%%%%%%%%%%%%%%%%%%%%%%%%%%%%%%%%%%%

\bibliographystyle{alpha}
\bibliography{Library}

\end{document}